\documentclass[a4paper,12pt,reqno]{amsart}
\usepackage{graphicx}

\usepackage{color}%for comments

\usepackage{amsmath,amstext,amssymb,amsopn,amsthm}
\usepackage{url}

\usepackage[margin=30mm]{geometry}
\usepackage{eucal,mathrsfs,dsfont}%gives nicer set-names. Use \mathds instead of \mathds
\usepackage{soul}

\newtheorem{theorem}{Theorem}[section]
\newtheorem{corollary}[theorem]{Corollary}
\newtheorem{lemma}[theorem]{Lemma}
\newtheorem{proposition}[theorem]{Proposition}

\theoremstyle{definition}

\newtheorem{definition}[theorem]{Definition}

\theoremstyle{remark}
\newtheorem{remark}[theorem]{Remark}

\newcommand{\eps}{\varepsilon}
\newcommand{\calA}{\mathcal{A}}
\newcommand{\Aa}{\mathcal{A}_{\alpha}}

\newcommand{\calG}{\mathcal{G}}

\newcommand{\calL}{\mathcal{L}}

\newcommand{\Ee}{\mathds{E}} % Expectation
\newcommand{\Pp}{\mathds{P}} % Probability
\newcommand{\R}{\mathds{R}}
\newcommand{\Rd}{\mathds{R}^d}
\newcommand{\N}{{\mathds{N}}}

\DeclareMathOperator{\dist}{dist}

\title[Harmonic functions for $x$-dependent rectilinear stable processes]{Harmonic functions on balls for $x$-dependent rectilinear stable processes}

\author[T. Kulczycki]{Tadeusz Kulczycki}
\author[M. Ryznar]{Micha{\l} Ryznar}
\thanks{T. Kulczycki and M. Ryznar were supported in part by the National Science Centre, Poland, grant no. 2023/49/B/ST1/03964}
\address{Faculty of Pure and Applied Mathematics, Wroc{\l}aw University of Science and Technology, Wyb. Wyspia{\'n}skiego 27, 50-370 Wroc{\l}aw, Poland.}
\email{tadeusz.kulczycki@pwr.edu.pl}
\email{michal.ryznar@pwr.edu.pl}

\begin{document}

\begin{abstract} We obtain sharp estimates for functions harmonic with respect to $x$-dependent rectilinear stable processes in balls, under the assumption that the Dirichlet exterior data are radial about the center. The main idea of the proof is based on the construction of global barrier functions for the $x$-dependent rectilinear fractional Laplacian in balls.
\end{abstract}

\maketitle

\section{Introduction}
Let $\alpha \in (0,2)$, $d \in \N$, $d \ge 2$ and $Z_t = (Z_t^{(1)},\ldots,Z_t^{(d)})^T$ be the rectilinear $\alpha$-stable process on $\R^d$, that is $Z_t^{(1)}, \ldots, Z_t^{(d)}$ are independent one-dimensional symmetric  $\alpha$-stable processes. We consider a stochastic differential equation
\begin{equation}
\label{main}
dX_t = A(X_{t-}) \, dZ_t, \quad X_0 = x \in \R^d, 
\end{equation}
where $A(x) = (a_{ij}(x))$ is a $d \times d$ matrix and there are constants $\eta_1, \eta_2 > 0$, such that for any $x \in \R^d$, $i, j \in \{1,\ldots,d\}$
\begin{equation*}
|a_{ij}(x)| \le \eta_1,
\end{equation*}
\begin{equation*}
\det(A(x)) \ge \eta_2,
\end{equation*}
and for any $i, j \in \{1,\ldots,d\}$ the functions $\R^d \ni x \to a_{ij}(x)$ are continuous.
 
It is well known (see \cite{BC2006}) that SDE (\ref{main}) has a unique weak solution $X$, which we call {\it the $x$-dependent rectilinear $\alpha$-stable process on $\R^d$}. It is a typical example of a pure jump process, whose L{\'e}vy measure depends on the position and is singular with respect to the Lebesgue measure on $\R^d$. 

Let $\Ee^x$, $\Pp^x$ denote respectively the expected value and the probability of the process $X$ starting from $x \in \R^d$. By $\tau_D = \inf\{t \ge 0:\, X_t \notin D \}$ we denote the first exit time of an open set $D \subset \R^d$ by $X$.

We say that a function $u$ that is Borel measurable and bounded in $\R^d$ is {\it harmonic} with respect to $X$ in a bounded open set $D \subset \R^d$ if 
$$
u(x) = \Ee^x(u(X_{\tau_D})) \quad \text{for every $x \in D$.}
$$
Clearly, if $D \subset \R^d$ is a bounded open set and $g: D^c \to \R$ is a bounded Borel function then $u$ given by
\begin{equation}
\label{Dirichlet}
u(x)= \left\{              
\begin{array}{lll}                   
\Ee^x(g(X_{\tau_D})),&\text{for}& x \in D, \\  
g(x),&\text{for}& x \in D^c
\end{array}       
\right. 
\end{equation}
is harmonic with respect to $X$ on $D$. The function $g$ may be treated as Dirichlet exterior data for $u$.

By Proposition 4.1 of \cite{BC2006}, the infinitesimal generator of the Markov process $X$ determined by (\ref{main}) is 
\begin{eqnarray}
\nonumber
\calL f(x) &=&  \Aa \sum_{i = 1}^d   \int_{\R \setminus \{0\}} \left[f(x + a_i(x) w) - f(x) - w 1_{\{|w| \le 1\}} \nabla f(x) a_i(x)\right] \, \frac{dw}{|w|^{1 + \alpha}},\\
\label{operator_L}
&=& \frac{\Aa}{2} \sum_{i = 1}^d   \int_{\R \setminus \{0\}} \left[f(x + a_i(x) w) + f(x - a_i(x) w) - 2f(x) \right] \, \frac{dw}{|w|^{1 + \alpha}},
\end{eqnarray}
where $\Aa = \alpha 2^{\alpha-1} \pi^{-1/2} \Gamma((1+\alpha)/2)/ \Gamma(1-\alpha/2)$ and  $a_i(x) = (a_{1i}(x),\ldots, a_{di}(x))$ is the $i$th column of $A(x)$. We call $\calL$ {\it the $x$-dependent rectilinear fractional Laplacian}. This operator is often presented in the following form
$$
\calL = -\sum_{i = 1}^d \left(-\partial_{a_i(x)}^2\right)^{\alpha/2},
$$
see e.g. (1.9) in \cite{FR2024}. $-\calL$ is a typical example of a nonlocal elliptic operator, which is not translation invariant and whose  
L{\'e}vy measure is singular with respect to the Lebesgue measure on $\R^d$ (see page 115 in \cite{FR2024book}). When  $A(x)$ is the identity matrix for each $x \in \R^d$ then $\calL$ is called {\it the rectilinear fractional Laplacian} (see \cite{CHZ2025}). Such operators have been extensively studied in recent years. It is well known that if $D \subset \R^d$ is an open set, $f:\R^d \to \R$ is a bounded Borel function and $f \in C_b^2(D)$ then $\calL f(x)$ is well defined for any $x \in D$. One may show (see Corollary \ref{Lharmonic_in_D}) that if $D$, $f$ are as above and additionally $D$ is bounded and $\calL f(x) = 0$ for any $x \in D$ then $f$ is harmonic with respect to $X$ on $D$. 

As usual, we denote a ball with center at $z \in \R^d$ and radius $r > 0$ by $B(z,r)$. For any set $D \subset \R^d$ we denote its complement by $D^c$. For any $r > 0$, we say that $g:B^c(z,r) \to \R$ is radial with respect to $z \in \R^d$ when there exists a radial profile function $\tilde{g}:[r,\infty) \to \R$ such that  for any $x \in B^c(z,r)$ we have $g(x) = \tilde{g}(|x-z|)$. 

The main result of this paper is the following theorem, which provides estimates for functions harmonic with respect to $X$ in balls, under the assumption that the Dirichlet exterior data are radial about the center.
\begin{theorem}
\label{main_thm}
Let $z \in \R^d$, $r > 0$, $D = B(z,r)$ and $g: D^c \to \R$ be a bounded Borel function, which is radial with respect to $z$. Denote by $\tilde{g}$ the radial profile function of $g$. Let $u$ be the function harmonic with respect to $X$ on $D$, which is given by (\ref{Dirichlet}). For each fixed $z \in \R^d$, $r > 0$ and $x \in D$ there exists a Borel function $f_{D}^x: [r,\infty) \to [0,\infty)$ such that 
\begin{equation*}
u(x) = \int_r^{\infty} f_{D}^x(y) \tilde{g}(y) \, dy.
\end{equation*}
Moreover, there exist constants $C_1, C_2 >0$ which depend only on $\alpha$, $d$, $\eta_1$, $\eta_2$ such that for any $y \in (r,\infty)$ we have 
\begin{equation*}
C_1 \varphi_D^x(y) \le f_{D}^x(y) \le C_2 \varphi_D^x(y) ,
\end{equation*}
where
\begin{equation}
\label{phi_formula}
\varphi_D^x(y) = \frac{\delta_D^{\alpha/2}(x) r^{\alpha/2}}{(y - r)^{\alpha/2} y^{\alpha/2} (y + \delta_D(x) - r)}
\end{equation}
and $\delta_D(x) = r - |x - z|$.
\end{theorem}

\begin{remark}
\label{probability_density}
For each fixed $z \in \R^d$, $r > 0$, $x \in D = B(z,r)$ and any Borel set $W \subset [r,\infty)$ put 
\begin{equation}
\label{kappa_definition}
\kappa_{D}^x(W) = \Pp^x(|X_{\tau_{D}} - z| \in W). 
\end{equation}
Clearly, $\kappa_{D}^x$ is a probability measure on $[r,\infty)$. The function $f_{D}^x$ is a density of $\kappa_{D}^x$ that is $\kappa_{D}^x(W) = \int_W f_{D}^x(y) \, dy$ for any Borel set $W \subset [r,\infty)$.
\end{remark}

As far as we know, Theorem \ref{main_thm} provides the first sharp two-sided estimates for harmonic functions associated with nonlocal elliptic operators which are not translation invariant and whose Lévy measures are singular with respect to the Lebesgue measure on $\R^d$. The novelty of the result lies in obtaining explicit boundary decay rates for harmonic functions of $x$-dependent rectilinear fractional Laplacians in balls under minimal regularity assumptions on the coefficients and under radial exterior data.
Although we focus on balls and on harmonic functions, the construction of global barrier functions developed in this paper is flexible and suggests possible extensions to other classes of domains and equations. 

The regularity theory of harmonic functions for operators $\calL$ was initiated by Bass and Chen in \cite{BC2010}. In that work, H\"older continuity of harmonic functions was established, and it was shown that the Harnack inequality fails in this setting.

Subsequent developments considered broader classes of $x$-dependent nonlocal elliptic operators, including operators with singular kernels. Within this more general framework, weak Harnack inequalities and H\"older regularity results were obtained in \cite{CK2020, DK2020}. In parallel, heat kernel estimates for such operators were studied extensively; see, for example, \cite{KSV2018, KKS2021, KKR2022, KKK2022, KKR2024} and the references therein. More recently, Schauder-type and Cordes--Nirenberg-type estimates for solutions of $x$-dependent nonlocal elliptic equations with singular kernels were proved in \cite{FR2024}, focusing on local regularity properties of solutions. For the (translation invariant) rectilinear fractional Laplacians, Dirichlet heat kernel estimates in $C^{1,1}$ domains were established in \cite{CHZ2025}.

Despite this substantial body of work, sharp two-sided estimates describing the boundary behaviour of harmonic functions for $x$-dependent nonlocal elliptic operators with singular L\'evy measures appear to be unavailable, even in simple geometric domains such as balls.

The remainder of the paper is organized as follows.
Section 2 introduces the notation and collects some facts that will be needed later. Section 3 is devoted to the proof of our main result, Theorem \ref{main_thm}. The crucial step, the verification that two auxiliary functions $f_{b_1,\theta}$ and $F_{b_2,\Theta}$ are, respectively, superharmonic and subharmonic with respect to the process $X$ in the ball $B(z,r)$, is deferred to Section 4. Finally, in the Appendix we construct certain one-dimensional $C^2$ functions that underlie the definition of $f_{b,\theta}$.

\section{Preliminaries} 
We may assume that $\alpha \in (0,2)$ and $d \in \N$, $d \ge 2$ are fixed.

For $a, b \in \R$ we denote $a \wedge b = \min(a,b)$, $a \vee b = \max(a,b)$. We denote $\N = \{1, 2, 3, \ldots\}$.

Unless stated otherwise we assume that all constants are positive and may depend only on $\alpha$, $d$, $\eta_1$, $\eta_2$. We adopt the convention that constants may change their value from one use to the next. We write $f_1(x) \approx f_2(x)$ for $x \in U$ when there exist constants $c_1, c_2 > 0$ depending only on $\alpha$, $d$, $\eta_1$, $\eta_2$ such that for all $x \in U$ we have $c_1 f_1(x) \le f_2(x) \le c_2 f_1(x)$. Similarly, we write $f_1(x) \lesssim f_2(x)$ for $x \in U$ when there exists a constant $c > 0$ depending only on $\alpha$, $d$, $\eta_1$, $\eta_2$ such that for all $x \in U$ we have $f_1(x) \le c f_2(x)$.

For $z \in \R^d$, $r > 0$, by \cite{R2025}, we have
\begin{equation}
\label{exit_time_ball}
\Ee^x\tau_{B(z,r)} \approx (r^2-|x-z|^2)^{\alpha/2}, \quad \text{for $x \in B(z,r)$}.
\end{equation}

By \cite{BC2006}, the family $(\Pp^x)_{x\in\R^d}$ defines a strong Markov process $X$ on $\R^d$.
Moreover, for each $x\in\R^d$, the measure $\Pp^x$ solves the martingale problem for $\calL$.
The process $X$ is shown to have a version with c\`adl\`ag paths \cite{BC2006}. For the rest of this paper we consider that version. When 
 $D \subset \R^d$ is a bounded domain, it is known that $\Pp^x(\tau_D < \infty) = 1$ for every $x \in D$. 
To verify quasi left continuity of $X$, we apply \cite[Theorem 3.20]{D1965}.
For this, it suffices to check that for any $r>0$,
$$\lim_{t\to 0^+} \sup_{x\in \R^d}\Pp^x\left(|X_{t}-x|\ge r\right) = 0.$$ This is a consequence of \cite[Cor. 4.4]{BC2006} which  reads
$$\Pp^x\left(\sup_{0\le s\le t}|X_{s}-x|\ge r\right)\le C \frac t{1\wedge r^2}, \quad x\in \R^d, \ t>0.$$
Therefore, by \cite[Theorem 3.20]{D1965}, the process $X$ is quasi left continuous. 

Let $\{e_1, \ldots, e_d\}$ be the standard basis in $\R^d$. The L{\'e}vy measure of $Z$ is given by 
$\mu(U) = \Aa \sum_{i = 1}^d \int_{\R} 1_U(w e_i) |w|^{-1-\alpha} \, dw$, where $U \subset \R^d$ is Borel measurable. The jumping measure of $X$ is given by $\nu(x,U) = \mu(\{v \in \R^d: \, x + A(x) v \in U\})$, where $x \in \R^d$ is a starting point, and $U \subset \R^d$ is Borel measurable.

By standard arguments (see e.g. \cite[Proposition 2.3]{BL2002} or \cite[pages 39-40]{CK2003}) for any Borel sets $U, V \subset \R^d$ such that $\dist(U,V) > 0$ 
\begin{equation*}
\sum_{s \le t} 1_{\{X_{s-} \in U, \, X_s \in V\}} - 
\int_0^t 1_U(X_s) \nu(X_s,V) ds
\end{equation*}
is a $\Pp^x$ martingale for each $x \in \R^d$. Let $z \in \R^d$, $r > 0$ and put $D = B(z,r)$. It follows that for any $t > 0$, $x \in D$, any Borel set $U \subset D$ and any Borel set $V \subset \Rd$ such that $\dist(V,D) > 0$ we have
\begin{equation}
\label{Levy_system_2}
\Ee^x\left(\sum_{s \le t \wedge \tau_D} 1_{\{X_{s-} \in U, \, X_s \in V\}}\right) = \Ee^x\left(\int_0^{t \wedge \tau_D} 1_U(X_s) \nu(X_s,V) ds \right).
\end{equation}
For any bounded open set $H \subset \R^d$ and $x \in H$ denote the Green measure $G_H(x,\cdot)$ by
$$
G_H(x,U) = \Ee^x\left(\int_0^{\tau_H} 1_U(X_s) \, ds\right), \quad \text{for any Borel set $U \subset H$.} 
$$
By taking $U = D$ and letting $t \to \infty$ in (\ref{Levy_system_2}) we obtain for any $x \in D$ and Borel set $V \subset \Rd$ such that $\dist(V,D) > 0$
\begin{equation}
\label{Levy_system_3}
\Pp^x(X_{\tau_D} \in V) = \int_D  \nu(y,V) G_D(x,dy).
\end{equation}

For any $v \in \R^d \setminus \{0\}$, $x \in \R^d$ and $f:\R^d \to \R$ put
\begin{equation}
\label{calL_v}
\calL_v f(x) = \frac{\Aa}{2}  \int_{\R \setminus \{0\}} \left[f(x + v w) + f(x - v w) - 2f(x) \right] \, \frac{dw}{|w|^{1 + \alpha}}.
\end{equation}
We have
$$
\calL_v f(x) = p.v. \Aa \int_{\R} \left[f(x + v w) - f(x) \right] \, \frac{dw}{|w|^{1 + \alpha}}.
$$
 
We will need the following auxiliary result.
\begin{lemma}
\label{lemma_L_e_d}
Assume that $r > 0$, $f: \R^d \to \R$ is Borel, bounded, radial and $f \in C^2(B(0,r))$. 

If for any $x \in B(0,r)$ we have $\calL_{e_d} f(x) \le 0$ then for any $x \in B(0,r)$ we have $\calL f(x) \le 0$ (where $\calL$ is given by (\ref{operator_L})). 

Similarly, if for any $x \in B(0,r)$ we have $\calL_{e_d} f(x) \ge 0$ then for any $x \in B(0,r)$ we have $\calL f(x) \ge 0$. 
\end{lemma}
\begin{proof}
Assume that for any $x \in B(0,r)$ we have $\calL_{e_d} f(x) \le 0$ (the proof in the case when $\calL_{e_d} f(x) \ge 0$ is the same). Since $f$ is radial, it is clear that for any $v \in \R^d \setminus \{0\}$ and any $x \in B(0,r)$ we obtain $\calL_{v/|v|} f(x) \le 0$. By standard arguments for $v \in \R^d \setminus \{0\}$ and $x \in B(0,r)$ we have $\calL_{v/|v|} f(x) = |v|^{-\alpha} \calL_{v} f(x) \le 0$. It follows that $\calL f(x) \le 0$ for $x \in B(0,r)$.
\end{proof}

Put
\begin{equation}
\label{lambda_formula}
\lambda(y)= \left\{              
\begin{array}{lll}                   
(r^2 - |y|^2)^{\alpha/2},&\text{for}& y \in B(0,r), \\  
0,&\text{for}& y \in B^c(0,r)
\end{array}       
\right. 
\end{equation}
For $x \in B(0,r)$ we have (\cite{R2025})
\begin{equation}
\label{lambda_formula_generator}
\calL_{e_d} \lambda(x) = -\tilde{\calA}_{\alpha},
\end{equation}
where $\tilde{\calA}_{\alpha}= 2^{\alpha} \pi^{-1/2} \Gamma((1+\alpha)/2) \Gamma(1+\alpha/2)$.

Put
\begin{equation}
\label{h_function}
h(y) =  \left\{              
\begin{array}{lll}                   
(r^2 - |y|^2)^{\alpha/2 - 1},&\text{for}& y \in B(0,r), \\  
0,&\text{for}& y \in B^c(0,r).
\end{array}
\right. 
\end{equation}
One can show that for any $x \in B(0,r)$ we have
\begin{equation}
\label{h_function_generator}
\calL_{e_d} h(x) = 0.
\end{equation}

Indeed, for $v>0$ let $s_v(u)= (v^2-|u|^2)^{\alpha/2-1}_+, u\in \R$, then it is well known that the function $s_v$ is $\alpha$-harmonic 
in the interval  $(-v,v)$, so 
\begin{equation}
\label{s_function_generator}
p.v.\int_{\R} \left[s_v(u +  w) - s_v(u)\right] \, \frac{dw}{|w|^{1 + \alpha}}= 0, \quad u\in  (-v,v).
\end{equation}
		
	Next, for $x=(x_1, x_2, \dots, x_d)$, $|x|<r	$ we have $|x_d|< \sqrt{r^2-|\tilde{ x}|^2}$, where $\tilde{ x}=x-x_de_d$ and, by \eqref{s_function_generator}, we obtain
		\begin{eqnarray}
\calL_{e_d} h(x)&=&p.v.\ {\Aa}\int_{\R} \left[(r^2-|x +  we_d|^2)^{\alpha/2-1}_+ - (r^2-|x|^2)^{\alpha/2-1}_+\right] \, \frac{dw}{|w|^{1 + \alpha}}\nonumber\\
&=&p.v.\ {\Aa}\int_{\R} \left[(r^2-|\tilde{ x}|^2-|x_d +  w|^2)^{\alpha/2-1}_+ - (r^2-|\tilde{ x}|^2-|x_d|^2)^{\alpha/2-1}_+\right] \, \frac{dw}{|w|^{1 + \alpha}}\nonumber\\
&=&p.v.\ {\Aa}\int_{\R} \left[s_{\sqrt{r^2-|\tilde{ x}|^2}}(x_d +  w) - s_{\sqrt{r^2-|\tilde{ x}|^2}}(x_d)
\right] \, \frac{dw}{|w|^{1 + \alpha}}= 0.\nonumber\end{eqnarray}
This implies \eqref{h_function_generator}.

\section{Proof of the main result}
In this section we present the proof of our main result, Theorem \ref{main_thm}. The key steps are Theorem \ref{Lharmonic_in_U} and Proposition \ref{main_thm_for_small_rings}. Theorem \ref{Lharmonic_in_U} is a localized Dynkin formula in martingale form (stopped at $\tau_U$). This result may be of independent interest. Proposition \ref{main_thm_for_small_rings} concerns estimates of harmonic functions with respect to $X$ when Dirichlet exterior data are indicators of rings. The key idea to obtain Proposition \ref{main_thm_for_small_rings} is the construction of global barrier functions $f_{b_1,\theta} - u$ and $u-F_{b_2,\Theta}$ with respect to the $x$-dependent rectilinear fractional Laplacian for the ball $D$. Part of the proof of Proposition \ref{main_thm_for_small_rings}, namely Proposition \ref{f_b_estimate1}, is deferred to Section 4.
\begin{theorem}
\label{Lharmonic_in_U}
If $D \subset \R^d$ is an open set, $U$ is open and bounded, $\bar U\subset D$ and a function $f:\R^d \to \R$ is bounded and Borel and $f \in C_b^2(D)$ then
$$
f(X_{t \wedge \tau_U}) - f(X_0) - \int_0^{t \wedge \tau_U} \calL f(X_s) \, ds
$$
is a martingale under $\Pp^x$ for each $x\in U$.
\end{theorem}
\begin{proof} We adapt the proof of Proposition 4.1 in \cite{BC2006}; the main additional point is to handle the merely local regularity assumption $f\in C_b^2(D)$, which requires a careful localization of the argument. Fix $x\in U$ and let $X$ be the weak solution to \eqref{main} under $\Pp^x$.
We apply the local It\^o formula (see, e.g., \cite[Lemma 3.8]{BKP2025}) to $f$, and obtain
\begin{eqnarray}
\nonumber
&& f(X_{t \wedge \tau_{U}} )-f(X_{0})\\
\nonumber
&=& 
\int_0^{t\wedge \tau_{U}}\nabla f(X_{u-})dX_u+\sum_{u\le t\wedge \tau_{U}, |\Delta X_u|>0 }f(X_{u} )-f(X_{{u-}})-\nabla f(X_{u-})\cdot\Delta X_u\\
\label{Ito}
&=& I(t)+J(t).\end{eqnarray}
%%%%%%%%%%%%%%%%%%%%%%%
We split $I(t)$ into the small jump part and the large jump part. Let $Z^1_t$ and $Z^2_t$ be two independent processes with L\'evy measures $\mu_1 = \mu|_{B_1}$ and $\mu_2=\mu-\mu_1$, respectively. Then we may assume that
$$Z_t= Z^1_t+Z^2_t \
\text{and}  \
X_t= X^1_t+X^2_t,
$$
where
%%%%%%%%%%%%%%%%%%%%
\begin{equation}
\label{main2}
dX^i_t = A(X_{t-}) \, dZ^i_t, \quad X_0 = x \in \R^d.
\end{equation}
It follows that 
$$I(t)= \int_0^{t\wedge \tau_{U}}\nabla f(X_{u-})dX^1_u+ \int_0^{t\wedge \tau_{U}}\nabla f(X_{u-})dX^2_u=I_1(t)+I_2(t).$$
Note that $Z_t^2$ is a compound Poisson process and its jumps are the big jumps of the process $Z_t$, hence 
%%%%%%%%%%%%%%%%%%%%%%%%%
 $$I_2(t) =\sum_{u\le t\wedge \tau_{U}, |\Delta Z_u|> 1 }\nabla f(X_{u-})\cdot\Delta X_u.$$
%%%%%%%%%%%%%%%%%%%%%%%%%%%%%%%%%%%%%%%%%%%%%%%%%%
Hence we can rewrite \eqref{Ito} as 
%%%%%%%%%%%%%%%%%%%%%%%%%%%%%%%%%%%%%%%%%%%%%%%%%%%%%%%%%%%%%%
\begin{eqnarray}\label{Ito1}f(X_{t \wedge \tau_{U}} )-f(X_{0})&=& 
I_1(t)+\sum_{u\le t\wedge \tau_{U}, |\Delta X_u|>0 }f(X_{u} )-f(X_{{u-}})-\nabla f(X_{u-})\cdot\Delta X_u \, 1_{\{|\Delta Z_u|\le 1\}}\nonumber\\&=& I_1(t)+J_1(t).\end{eqnarray}
%%%%%%%%%%%%%%%%%%%%%%%%%%%%%%%%%%%%%%%%%%%%%%%%%%%%%%%%%%
We claim that $I_1(t)$ is a martingale. This follows from the fact that its quadratic variation process $[I_1, I_1]_t$ is integrable (see e.g. \cite[Corollary 3, p.73]{P2005} ). Note that on the event  $u\le t\wedge \tau_{U} $
$$|\Delta X_u|=  |A(X_{u-}) \Delta Z_u|\le \sup_{x\in U} |A(x)||\Delta Z_u| $$  and
$$|\nabla f(X_{u-})|\le \sup_{x\in U}|\nabla f(x)|.$$
Hence,
\begin{eqnarray*}\Ee^x[I_1, I_1]_t&\le& \Ee^x\sum_{u\le t\wedge \tau_{U}, |\Delta Z_u|\le 1 } \max_{v\le t\wedge \tau_{U}}(|\nabla f(X_{v-})|)^2|\Delta X_u|^2\\
&\le& C \Ee^x\sum_{u\le t, |\Delta Z_u|\le 1 } |\Delta Z_u|^2\\
&=& Ct\int_{|x|\le 1} |x|^2\mu(dx)<\infty, \end{eqnarray*}
 which proves integrability of quadratic variation process.

%where 
%$$A_1 = \int_0^{t\wedge \tau_{r}}\nabla S(X_{u-})dX_u,$$
%$$A_2 =\sum_{u\le t\wedge \tau_{r}, 0<|\Delta X_u| }S(X_{u} )-S(X_{{u-}})-\nabla S(X_{u-})\cdot\Delta X_u I_{|\Delta X_u|\le 1}$$
%and
%$$A_3 =-\sum_{u\le t\wedge \tau_{r}, |\Delta X_u|> 1 }\nabla S(X_{u-})\cdot\Delta X_u$$
%Next, on the event  $\{u\le t\wedge \tau_{r}\}$ we have  $|\Delta X_u|\le 2$, hence 
%%%%%%%%%%%%%%%%%%%%%%%%%
Next, we observe that $J_1(t)$ can be rewritten as 
%%%%%%%%%%%%%%%%
$$J_1(t)=\sum_{u\le t\wedge \tau_{U}, |\Delta Z_u|>0 }f(X_{u-}+ A(X_{u-}) \Delta Z_u  )-f(X_{{u-}})-\nabla f(X_{u-})\cdot A(X_{u-}) \Delta Z_u \, 1_{\{|\Delta Z_u|\le 1\}}. $$
Since the coordinates cannot jump at the same time we can split $J_1(t)$ into $d$ parts
$$J_1(t)=\sum_{j=1}^d J^{j}_1(t),$$
where 
$$J^{j}_1(t)= \sum_{u\le t }g(u)\left[f(X_{u-}+ a_j(X_{u-}) \Delta Z^{(j)}_u  )-f(X_{{u-}})-\nabla f(X_{u-})\cdot a_j(X_{u-}) \Delta Z^{(j)}_u \, 1_{\{|\Delta Z^{(j)}_u|\le 1\}}\right] $$
with  $g(u)= 1_{\{u\le  \tau_{U}\}}$. 
%%%
Let, for $w\in \R, u>0$,
$$F_j(w,u)= g(u)\left[f(X_{u-}+ a_j(X_{u-})w  )-f(X_{{u-}})-\nabla f(X_{u-})\cdot a_j(X_{u-}) w 1_{\{|w|\le 1\}}\right].$$
Note that $g(u)$ is  a predictable process. Hence $F_j(w,u)$
 is  jointly measurable with respect to the product of Borel $\sigma$-field and a predictable 
$\sigma$-field generated by the process $Z^{(j)}$.  Note that 
$$\Ee\int_0^t \int_{\R}|F_j(w,u)|\frac 1{|w|^{1+\alpha}}dw<\infty.$$ This is justified by the following estimate
%%%
$$|F_j(w,u)|\le  C \min(|w|^2,1),$$
where $c=c(U,f)$. This follows from the fact that $\bar{U}\subset D$ and $f \in C_b^2(D)$.
%$$\left|f(x+y)-f(x)-\nabla f(x)\cdot y I_{|y|\le 1}\right|\le c \min(|y|^2,1)$$ if $x\in U $,
%where $c=c(U)$.

%Next, using L\'evy's system formula we have that
By the compensation formula
$$M_j(t)= F_j(\Delta Z^{(j)}_t,t)- \int_0^t \int_{\R}F_j(w,u)\frac{\Aa }{|w|^{1+\alpha}}dw du$$
is a martingale.

Next, note that $$M_j(t)= J^{j}_1(t)- \int_0^{t\wedge \tau_{U}}\calL_j  f(X_{{u}})du,$$
where
%%%%%%%%%%%%%%%%%%
$$\calL_j  f(x)= \Aa \int_{\R \setminus \{0\}} \left( f(x+a_j(x)w)-f(x)-\nabla f(x)\cdot a_j(x)w  
 1_{\{|w|\le1\}}\right)\frac{ dw}{|w|^{1+\alpha}}.$$ 

Summing up we see that $$J_1(t)-\int_0^{t\wedge \tau_{U}}\calL  f(X_{{u}}) du$$ is a martingale.
%%%%%%%%%%%%%%%%%%%%
Since $I_1(t)$ is a martingale, we arrive at the conclusion of the theorem.
\end{proof}

We say that a function $u$ that is Borel measurable and bounded in $\R^d$ is {\it superharmonic} with respect to $X$ in a bounded open set $D \subset \R^d$ if 
\begin{equation}
\label{superharmonic}
u(x) \ge \Ee^x(u(X_{\tau_D})) \quad \text{for every $x \in D$}
\end{equation} 
and is {\it subharmonic} if
\begin{equation}
\label{subharmonic}
u(x) \le \Ee^x(u(X_{\tau_D})) \quad \text{for every $x \in D$}.
\end{equation}

By Theorem \ref{Lharmonic_in_U} we obtain
\begin{corollary}
\label{Lharmonic_in_D}
Assume that $D \subset \R^d$ is a bounded open set, a function $f:\R^d \to \R$ is bounded and Borel and $f \in C_b^2(D)$. Then we have

(i) if $\calL f(x) = 0$ for any $x \in D$ then $f$ is harmonic with respect to $X$ on $D$, 

(ii) if $\calL f(x) \ge 0$ for any $x \in D$ then $f$ is subharmonic with respect to $X$ on $D$, 

(iii) if $\calL f(x) \le 0$ for any $x \in D$ then $f$ is superharmonic with respect to $X$ on $D$. 
\end{corollary}
\begin{proof}
We prove only (i), the proofs of (ii) and (iii) are similar and are omitted.

For any $n \in \N$ put $D_n = \{x \in D: \, \dist(x, \partial D) > 1/n\}$. For sufficiently large $n$ sets $D_n$ are not empty. For such $n$ by Theorem \ref{Lharmonic_in_U} we have
\begin{equation}
\label{Lharmonic_in_D_n}
f(x) = \Ee^x(f(X_{\tau_{D_n}})) \quad \text{for every $x \in D_n$.}
\end{equation}

Recall that $\tau_D < \infty$ a.s. Note that $\tau_{D_n} \le \tau_D$. Put $T = \lim_{n \to \infty} \tau_{D_n}$.  By quasi left continuity of the process $X$ we obtain $\lim_{n \to \infty} X_{\tau_{D_n}} = X_{T}$ a.s. Hence $X_{T} \in D^c$ a.s., so $T = \tau_D$ a.s. It follows that $\lim_{n \to \infty} X_{\tau_{D_n}} = X_{\tau_D}$ a.s. Using this and (\ref{Lharmonic_in_D_n}) we get
$$
f(x) = \Ee^x(f(X_{\tau_D})) \quad \text{for every $x \in D$,}
$$
so $f$ is harmonic with respect to $X$ on $D$.
\end{proof}

By (\ref{superharmonic}) and (\ref{subharmonic}) we obtain the following maximum principle.
\begin{corollary}
\label{maximum_principle}
Assume that $D \subset \R^d$ is a bounded open set, $f:\R^d \to \R$ is bounded and Borel. If $f$ is subharmonic with respect to $X$ on $D$ then we have
$$
f(x) \le \sup\{f(y): \, y \in D^c\} \quad \text{for every $x \in D$,}
$$
If $f$ is superharmonic with respect to $X$ on $D$ then we have
$$
f(x) \ge \inf\{f(y): \, y \in D^c\} \quad \text{for every $x \in D$.}
$$
\end{corollary}

Now, fix $z \in \R^d$, $r > 0$ and put $D = B(z,r)$.
For any $r > 0$ and $\eps \in (0,r/4]$ let us define the following auxiliary class of functions.  
\begin{definition}
\label{Class_A_1}
We define $\calG(r,\eps)$ to be the set of all functions $\theta: \, [0,r^2) \to (0,\infty)$, which satisfy the following conditions. 

(i) \begin{equation*}
\theta(v)= \left\{              
\begin{array}{lll}                   
\displaystyle{\frac{1}{r^2 - v}},&\text{for}& v \in [0,(r-\eps)^2], \\  
C_0,&\text{for}& v \in [(r-\eps +K)^2,r^2),
\end{array}       
\right. 
\end{equation*}
for some $K \in (0,\eps/4]$ and $C_0 >0$, 

(ii) $\theta \in C^2[0,r^2)$,

(iii) $\theta$ is strictly increasing on $[(r-\eps)^2, (r-\eps +K)^2)$,

(iv) $\max\{\theta''(v): \, v \in [0,r^2)\} = \theta''((r-\eps)^2)$. 
\end{definition}
A construction of functions in $\calG(r,\eps)$ is given in the Appendix.

For any fixed $r > 0$ and $\eps \in (0,r/4]$ let us define the following auxiliary function.
\begin{equation} \label{Theta}
\Theta(v)= \left\{              
\begin{array}{lll}                   
\displaystyle{\frac{1}{r^2 - v}},&\text{for}& v \in [0,(r-\eps)^2], \\  
\frac{1}{r^2 - v}\left[1-\left(\frac{v-(r-\eps)^2}q\right)^3\right],&\text{for}& v  \in ((r-\eps)^2,r^2), 
\end{array}
\right.       
\end{equation}
where $q=r^2- (r-\eps)^2$.
It is straightforward to check that $\Theta \in C^2[0,r^2)$.

\begin{lemma}
\label{theta_estimates}
Let $\theta \in \calG(r,\eps)$. Then for any $v \in [0,r^2)$ we have
$$
\theta'(v) \le \frac{3}{q^2}, \quad \Theta'(v) \le \frac{3}{q^2}, \quad \theta(v) \le \frac{4}{q}, \quad \Theta(v) \le \frac{4}{q}.
$$
\end{lemma}
\begin{proof}
We will show the results for $\theta \in \calG(r,\eps)$, the corresponding bounds for $\Theta$ are straightforward. For $v \in [0,(r - \eps)^2]$ we clearly have $\theta'(v) = (r^2 - v)^{-2} \le 1/q^2$ and $\theta(v) = (r^2 - v)^{-1} \le 1/q$. For $v \in ((r - \eps)^2, r^2)$ we have
$$
\theta'(v) = \theta'((r - \eps)^2) + \int_{(r - \eps)^2}^v \theta''(t) \, dt \le 
\frac{1}{q^2} + \int_{(r - \eps)^2}^{r^2} \frac{2}{q^3} \, dt = \frac{3}{q^2}
$$
and
$$
\theta(v) = \theta((r - \eps)^2) + \int_{(r - \eps)^2}^v \theta'(t) \, dt \le 
\frac{1}{q} + \int_{(r - \eps)^2}^{r^2} \frac{3}{q^2} \, dt = \frac{4}{q}.
$$ \end{proof}

It follows that there exist absolute constants $c_1, c_2 > 0$ such that for any $\theta \in \calG(r,\eps)$ and $\Theta$  and all $v \in [0,r^2)$ we have
\begin{equation}
\label{theta_estimate}
\frac{c_1}{(r^2 - v) \vee q} \le \theta(v) \le \frac{c_2}{(r^2 - v) \vee q},
\end{equation}
\begin{equation}
\label{Theta_estimate}
\frac{c_1}{(r^2 - v) \vee q} \le \Theta(v) \le \frac{c_2}{(r^2 - v) \vee q}.
\end{equation}

\begin{proposition}
\label{main_thm_for_small_rings}
Let $\eps \in (0,r/4]$ and $\eta \in (0,\eps]$ and define the function $g:\Rd\mapsto \R$ by 
\begin{equation}
\label{g_definition}
g(y)= \left\{              
\begin{array}{lll}                   
1,&\text{for}& |y - z| \in (r+\eps,r+\eps+\eta), \\  
0, &\text{otherwise.}&
\end{array}       
\right. 
\end{equation}
Let $u$ be the function harmonic with respect to $X$ in a ball $D = B(z,r)$, which is given by (\ref{Dirichlet}). Then there exist constants $C_1, C_2$ such that for any $x \in D$ we have
\begin{equation}
\label{harmonic_measure_for_small_rings}
C_1 \int_{r+\eps}^{r+\eps+\eta} \varphi_D^x(y) \, dy \le u(x) \le C_2 \int_{r+\eps}^{r+\eps+\eta} \varphi_D^x(y) \, dy
\end{equation}
where $\varphi_D^x$ is given by (\ref{phi_formula}).
\end{proposition}
\begin{proof} 
Without loss of generality we may assume that $z = 0$, so $D = B(0,r)$. Fix $\eps \in (0,r/4]$ and $\eta \in (0,\eps]$.

Let for a given function $\theta \in \calG(r,\eps)$
\begin{equation*}
f_{\theta}(x)= \left\{              
\begin{array}{lll}                   
\displaystyle{ (r^2 - |x|^2)^{\alpha/2} \theta(|x|^2)},&\text{for}& x \in B(0,r), \\  
0,&\text{for}& x \in B^c(0,r).
\end{array}       
\right. 
\end{equation*}

Now for any $b > 0$ and any function $\theta \in \calG(r,\eps)$ let us define an auxiliary function $f_{b,\theta}$ by the formula
\begin{equation}
\label{f_b_definition}
f_{b,\theta}(x)= \frac{b \eta r^{1- \alpha/2}}{\eps^{\alpha/2}}f_{\theta}(x)+g(x),\ x\in\R^d.                 
\end{equation}

Next, for the already defined  function $\Theta$, let
\begin{equation}
\label{f_definition2}
F_{\Theta}(x)= \left\{              
\begin{array}{lll}                   
\displaystyle{ (r^2 - |x|^2)^{\alpha/2} \Theta(|x|^2)},&\text{for}& x \in B(0,r), \\  
0,&\text{for}& x \in B^c(0,r).
\end{array}       
\right. 
\end{equation}

Similarly, for any $b > 0$  let us define an auxiliary function $F_{b,\Theta}$ by the formula

\begin{equation}
\label{F_b_definition}
F_{b,\Theta}(x)= \frac{b \eta r^{1- \alpha/2}}{\eps^{\alpha/2}}F_{\Theta}(x)+g(x),\ x\in\R^d.                 
\end{equation}
%\begin{equation}
%\label{F_b_definition}
%F_{b,\Theta}(x)= \left\{              
%\begin{array}{lll}                   
%\displaystyle{\frac{b \eta r^{1- \alpha/2}}{\eps^{\alpha/2}} (r^2 - |x|^2)^{\alpha/2} \Theta(|x|^2)},&\text{for}& x \in B(0,r), \\  
%g(x),&\text{for}& x \in B^c(0,r).
%\end{array}       
%\right. 
%\end{equation}

Clearly, for any $b > 0$  the functions $f_{b,\theta}$, $F_{b,\Theta}$ are bounded, Borel on $\R^d$ and $f_{b,\theta}, F_{b,\Theta} \in C_b^2(B(0,r))$. 

The key element of the proof of Theorem \ref{main_thm} is the following result.
\begin{proposition}
\label{f_b_estimate1}
Put $z = 0$, fix $r > 0$, $\eps \in (0,r/4]$ and $\eta \in (0,\eps]$. There exists a function $\theta \in \calG(r,\eps)$ (depending only on $r$, $\eps$ and $\alpha$)  and there exists a constant $b_1 > 0$ (depending only on $\alpha$) such that
$$
\calL f_{b_1,\theta}(x) \le 0 \quad \text{for any $x \in B(0,r)$.}
$$
There exist a constant $b_2 > 0$ (depending only on $\alpha$) such that
$$
\calL F_{b_2,\Theta}(x) \ge 0 \quad \text{for any $x \in B(0,r)$.}
$$
\end{proposition}

The proof of Proposition \ref{f_b_estimate1} is deferred to Section 4.
By Corollary \ref{Lharmonic_in_D} we obtain that $f_{b_1,\theta}$ is superharmonic with respect to $X$ in $B(0,r)$ and 
$F_{b_2,\Theta}$ is subharmonic with respect to $X$ in $B(0,r)$. Put $u_1 = f_{b_1,\theta} - u$ and $u_2 = F_{b_2,\Theta} - u$. It follows that $u_1$ is superharmonic with respect to $X$ in $B(0,r)$ and $u_2$ is subharmonic with respect to $X$ in $B(0,r)$. Note also that $u_1 \equiv u_2 \equiv 0$ on $B^c(0,r)$. In other words, $u_1$ and $-u_2$ are global barrier functions with respect to $x$-dependent rectilinear fractional Laplacian for the ball $B(0,r)$. By Corollary \ref{maximum_principle} $u_1 \ge 0$ on $B(0,r)$ and $u_2 \le 0$ on $B(0,r)$. Therefore we obtain
\begin{equation}
\label{f_b_1}
f_{b_1,\theta}(x) \ge u(x) \quad \text{for every $x \in B(0,r)$}
\end{equation}
and
\begin{equation}
\label{f_b_2}
F_{b_2,\Theta}(x) \le u(x) \quad \text{for every $x \in B(0,r)$.}
\end{equation}

By (\ref{phi_formula}) for $x \in B(0,r)$ we have
\begin{equation}
\label{phi_r_x_est}
\int_{r+\eps}^{r+\eps+\eta} \varphi_D^x(y) \, dy \approx \frac{\delta_D^{\alpha/2}(x) \eta}{\eps^{\alpha/2} (\eps \vee \delta_D(x))}.
\end{equation}

Note that for $x \in B(0,r)$ we have
$$
\delta_D^{\alpha/2}(x) \approx \frac{(r^2 - |x|^2)^{\alpha/2}}{r^{\alpha/2}}
$$
and
$$
\frac{1}{\delta_D(x) \vee \eps} = \frac{r}{(r(r-|x|)) \vee (r \eps)} \approx \frac{r}{(r^2 - |x|^2) \vee q}.
$$
Using this and (\ref{theta_estimate}), (\ref{Theta_estimate}) we obtain that for $x \in B(0,r)$
$$
\frac{1}{\delta_D(x) \vee \eps} \approx r \theta(|x|^2) \approx r \Theta(|x|^2).  
$$
By the above estimates we get that for $x \in B(0,r)$
$$
f_{b,\theta}(x) = \frac{b \eta r^{1- \alpha/2} (r^2 - |x|^2)^{\alpha/2} \theta(|x|^2)}{\eps^{\alpha/2}} 
\approx \frac{b \eta \delta_D^{\alpha/2}(x)}{\eps^{\alpha/2} (\eps \vee \delta_D(x))}.
$$
Similar estimate holds for $F_{b,\Theta}$. By this, (\ref{phi_r_x_est}), (\ref{f_b_1}), (\ref{f_b_2}) and the fact that constants $b_1$, $b_2$ depend only on $\alpha$ we obtain that there exist constants $C_1, C_2$ such that
$$
C_1 \int_{r+\eps}^{r+\eps+\eta} \varphi_D^x(y) \, dy \le u(x) \le C_2 \int_{r+\eps}^{r+\eps+\eta} \varphi_D^x(y) \, dy \quad \text{for every $x \in B(0,r)$.}
$$
\end{proof}

\begin{lemma}\label{large} Let $z_0\in \Rd$ and let $y\in \Rd$ be such that $|y-z_0|< r$ and $0<\eta<r<(4/5)R$, then
we have $$\mu(\{w \in \R^d: \,  R\le|y + A(y) w- z_0|\le R+\eta\})\approx \sum_{i=1}^d|a_i(y)|^\alpha   \frac \eta{R^{1+\alpha}}.$$ 
\end{lemma}
\begin{proof} We may assume that $z_0=0$. Fix $i\in \{1,2,\dots,d\}$. Let  $ a= a_i= a_i(y)$ and $v$ be real.
We would like to solve $|y+av|=R$.  This is equivalent to the following square equation  $|y+av|^2=R^2$ which reads

$$|y|^2+2v(y\cdot a)+|a|^2v^2=R^2.$$
Denote $\lambda^2 = |y|^2- (a/|a|\cdot y)^2$.
Then we can write the positive solutions as  
\begin{equation}
\label{vplus}
v_{+}(R)=\frac1{|a|}\left( \sqrt {R^2-\lambda^2 }- (a/|a|\cdot y)\right),
\end{equation}
while the negative solution is
\begin{equation}
\label{vminus}
v_{-}(R)=\frac1{|a|}\left( -\sqrt {R^2-\lambda^2 }- (a/|a|\cdot y)\right).
\end{equation}
Hence
\begin{eqnarray*}v_{+}(R+\eta)-v_{+}(R)&=&\frac1{|a|}\left(\sqrt {(R+\eta)^2-\lambda^2 }- \sqrt {R^2-\lambda^2 })\right)\\
&=&\frac1{|a|}\frac {(R+\eta)^2-R^2}{\sqrt {(R+\eta)^2-\lambda^2 }+ \sqrt {R^2-\lambda^2 })}
\end{eqnarray*}
Observe that for $|y|\le r\le (4/5)R$ we have 
$$\frac35(R+\eta)\le \sqrt{(R+\eta)^2-\lambda^2 }+ \sqrt {R^2-\lambda^2 }\le (R+\eta)+R.$$
which implies 
$$\frac\eta {|a|}\le v_{+}(R+\eta)-v_{+}(R)\le 4\frac\eta {|a|}. $$ 
Similarly 
$$\frac\eta {|a|}\le v_{-}(R)-v_{-}(R+\eta)\le 4\frac\eta {|a|}. $$ 
Next, if 
$R\le |y+av|\le R+\eta$, by the triangle inequality, we have
$$\frac {R-r}{|a|} \le |v|\le \frac{R+ 2r}{|a|}.$$
Let $P_i= \{v: \, R\le  |y +va_i|\le R+\eta \}= [v_{-}(R+\eta), v_{-}(R)]\cup [v_{+}(R),v_{+}(R+\eta)]$. Using the above inequalities
we obtain
$$\int_{P_i}\frac1{|v|^{1+\alpha}}dv\approx |a_i|^\alpha \frac \eta{R^{1+\alpha}}$$
which yields that
$$\mu(\{w \in \R^d: \, R\le|y +wA(y)|\le R+\eta\})  = \Aa\sum_{i = 1}^d\int_{P_i}\frac1{|v|^{1+\alpha}}dv
\approx  \sum_{i=1}^d|a_i|^\alpha  \frac\eta{R^{1+\alpha}}.$$
\end{proof}

\begin{lemma}\label{I-W} Let $z_0\in \Rd$ and let $x\in \Rd$ be such that $|x-z_0|< r$ and $0<\eta<r<(4/5)R$, then
$$\Pp^x(R\le|X_{\tau_{B(z_0,r)}}-z_0|\le R+\eta)\approx (r^2-|x-z_0|^2)^{\alpha/2}\frac \eta{R^{1+\alpha}}.$$ 
\end{lemma}
\begin{proof} We may assume $z_0 = 0$. Denote $D= B(0,r)$ and $P= \{w\in \Rd: \, R\le |w|\le R+\eta\}.$ By (\ref{Levy_system_3})
we have 
$$\Pp^x(X_{\tau_{D}}\in P)=\int_D \nu(y,P) G_D(x, dy).$$
By  Lemma \ref{large} we have that for all $y\in D$
$$C_2\frac \eta{R^{1+\alpha}}\le \nu(y,P)\le C_1 \frac \eta{R^{1+\alpha}}.$$
Using this and  (\ref{exit_time_ball}) we get
$$
\Pp^x(X_{\tau_{D}}\in P)\approx \int_ D \frac \eta{R^{1+\alpha}} G_D(x, dy) 
= \frac \eta{R^{1+\alpha}} \Ee^x(\tau_D) 
\approx (r^2-|x|^2)^{\alpha/2} \frac \eta{R^{1+\alpha}},
$$
which gives the assertion of the lemma.
\end{proof}

The following lemma is crucial for the proof of Proposition \ref{harmonic_measure_boundary}.
\begin{lemma} \label{large1}Let $z_0\in \Rd$ and  let $x\in \Rd$ be such that $|x-z_0|< R$ and 
$r_x=\frac13(R-|x-z_0|)$. There exists a constant $c$ such that for $y\in B(x,r_x)$
we have 
$$
\mu(\{w \in \R^d: \, |y-z_0 +wA(y)|\ge R\})\ge   \frac{c}{(R-|x - z_0|)^{\alpha}}.
$$ 
\end{lemma}
\begin{proof} We may assume that $z_0=0$. Fix $i\in \{1,2,\dots,d\}$. Let  $y\in B(x,r_x)$, $a= a_i= a_i(y)$, $v$ be real and put 
$\lambda^2 = |y|^2-(a/|a|\cdot y)^2$.
Then as in the proof of Lemma \ref{large} we have that 
$$ Q_i= \{v: |y +a_i v|\ge R\}= (-\infty, v_{-}(R)]\cup [v_{+}(R), \infty),$$
where $v_{+}(R)$, $v_{-}(R)$ are given by (\ref{vplus}), (\ref{vminus}) respectively.
Let $$v(R)= | v_{-}(R)| \wedge | v_{+}(R)|= \frac1{|a|}\frac {R^2-|y|^2}{ \sqrt {R^2-\lambda^2 }+|(a/|a|\cdot y)|}.  $$
Then 
$$\int_{Q_i}\frac1{|v|^{1+\alpha}}dv\ge \int_{v(R)}^\infty\frac1{|v|^{1+\alpha}}dv=\frac1{\alpha (v(R))^\alpha}= |a|^\alpha  \frac { (\sqrt {R^2-\lambda^2 }+|(a/|a|\cdot y)|)^\alpha}{\alpha(R^2-|y|^2)^\alpha}.$$

First we consider the case $|x|>(1/2)R$. Then we have $R-|y|\le (4/3) (R-|x|)$ which implies that
$$R^2-|y|^2\le 2R(R-|y|)\le \frac83R(R-|x|).$$
This together with
$$(\sqrt {R^2-\lambda^2 }+|(a/|a|\cdot y)|)^\alpha\ge |(a/|a|\cdot y)|^\alpha$$ and $|y|>(1/3)R$
yields
\begin{equation}\label{w_est}\int_{Q_i}\frac1{|v|^{1+\alpha}}dv\ge \left(\frac38\right)^\alpha   \frac { |y|^\alpha|(a\cdot y/|y|)|^\alpha}{\alpha R^\alpha(R-|x|)^\alpha}\ge \left(\frac18\right)^\alpha   \frac { |(a\cdot y/|y|)|^\alpha}{\alpha (R-|x|)^\alpha}. \end{equation}

In the case $|x|\le(1/2)R$ we have $R\le 2(R-|x|)$,  so
\begin{equation}\label{w_est1}\int_{Q_i}\frac1{|v|^{1+\alpha}}dv\ge  |a|^\alpha  \frac { (\sqrt {R^2-\lambda^2 }+|(a/|a|\cdot y)|)^\alpha}{\alpha(R^2-|y|^2)^\alpha}\ge  \left( \frac12\right)^{\alpha} \frac {|a|^\alpha}{\alpha (R-|x|)^{\alpha}}.\end{equation}

By \eqref{w_est} and \eqref{w_est1} we obtain 
$$\mu(\{w \in \R^d: \, |y +A(y)w|\ge R\})  = \Aa\sum_{i=1}^d\int_{Q_i}\frac1{|v|^{1+\alpha}}dv
\ge     \frac{c}{(R-|x|)^\alpha},$$
where  $c=\frac\Aa\alpha \left(\frac18\right)^\alpha \min_{y \in \R^d,|z|=1}\sum_{i=1}^d  { |(a_i(y)\cdot z)|^\alpha}>0$. Note that positivity of $c$ and its dependence only on $\alpha, d, \eta_1, \eta_2$ follows from the assumptions on the family of matrices $A(y), y\in \Rd$. 
\end{proof}

\begin{corollary}
\label{Bxp}
Fix $z \in \R^d$, $r > 0$ and put $D = B(z,r)$. For $x \in D$ let $r_x = \frac{1}{3}\dist(x,D^c)$ and $B_x = B(x,r_x)$. There exists a constant $p$ such that for any $x \in D$ we have
$$
\Pp^x(X_{\tau_{B_x}} \in D^c) \ge p.
$$ 
\end{corollary}
\begin{proof}
Let $x \in D$. By (\ref{Levy_system_3}) we have
\begin{equation}
\label{IWBx}
\Pp^x(X_{\tau_{B_x}} \in D^c) = \int_{B_x} \nu(y,D^c) G_{B_x}(x,dy).
\end{equation}
By Lemma \ref{large1} for $y \in B_x$ we get
\begin{eqnarray*}
\nu(y,D^c) &=& \mu(\{w \in \R^d: \, y + A(y) w \in D^c\})\\
&=& \mu(\{w \in \R^d: \, |y + w A(y) - z| \ge r\})\\
&\ge& \frac{c}{(r - |x - z|)^{\alpha}}.
\end{eqnarray*}
Using this, (\ref{IWBx}) and (\ref{exit_time_ball}) we get
\begin{eqnarray*}
\Pp^x(X_{\tau_{B_x}} \in D^c) &\ge& 
\frac{c}{(r - |x - z|)^{\alpha}} \int_{B_x} G_{B_x}(x,dy)\\
&=& \frac{c \Ee^x(\tau_{B_x})}{(\dist(x,D^c))^{\alpha}}\\
&\approx& \frac{r_x^{\alpha}}{(\dist(x,D^c))^{\alpha}},
\end{eqnarray*}
which gives the assertion.
\end{proof}

\begin{proposition}
\label{harmonic_measure_boundary}
Let $z \in \R^d$, $r > 0$ and put $D = B(z,r)$. Then for any $x \in D$ we have
$$
\Pp^x(X_{\tau_{D}} \in \partial D) = 0.
$$
\end{proposition}
The proof of this proposition uses arguments similar to the ones used in proofs of Lemma 6 in \cite{B1997} and Lemma 2.10 in \cite{KS2006}. For the convenience of the reader we repeat these arguments.
\begin{proof}
For $x \in D$ put $r_x = \frac{1}{3}\dist(x,D^c)$ and $B_x = B(x,r_x)$. By the strong Markov property we get for $x \in D$
\begin{equation}
\label{p0r0}
\Pp^x(X_{\tau_{D}} \in \partial D) =
\Pp^x(X_{\tau_{B_x}} \in \partial D) + 
\Ee^x\left(P^{X_{\tau_{B_x}}}(X_{\tau_{D}} \in \partial D)], \, X_{\tau_{B_x}} \in D\right).
\end{equation}
We denote the two terms on the right-hand side of (\ref{p0r0}) by $p_0(x)$ and $r_0(x)$ respectively. Using (\ref{p0r0}) we can prove inductively that for $k = 0, 1, \ldots$ and $x \in D$ we have
\begin{equation}
\label{pkrk}
\Pp^x(X_{\tau_{D}} \in \partial D) = p_0(x) + p_1(x) + \ldots + p_{k}(x) + r_k(x),
\end{equation}
where
\begin{equation}
\label{pk1}
p_{k + 1}(x) = \Ee^x\left(p_k(X_{\tau_{B_x}}), \, X_{\tau_{B_x}} \in D\right),
\end{equation}
and
$$
r_{k + 1}(x) = \Ee^x\left(r_k(X_{\tau_{B_x}}), \, X_{\tau_{B_x}} \in D\right).
$$
By Corollary \ref{Bxp}
$$
\sup_{x \in D} r_{k + 1}(x) \le (1 - p) \sup_{x \in D} r_{k}(x) \le (1- p)^{k + 1} \to 0, \quad \text{as $k \to \infty$.}
$$
By (\ref{pkrk}) we get $\displaystyle \Pp^x(X_{\tau_{D}} \in \partial D) = \sum_{k = 0}^{\infty} p_k(x)$. 
Using
 (\ref{Levy_system_3}) we obtain $p_{0}(x) = \Pp^x(X_{\tau_{B_x}} \in \partial D) = 0$. 
By (\ref{pk1}) we get $p_k(x) = 0$ for all $x \in D$ and all $k$, which finishes the proof. 
\end{proof}

We are now ready to prove our main result.
\begin{proof}[proof of Theorem \ref{main_thm}]
Fix $x \in D = B(z,r)$. We will show that there exist constants $C_1, C_2$ such that for any Borel set $W \subset [r,\infty)$ we have
\begin{equation}
\label{kappa_est}
C_1 \int_W \varphi_D^x(y) \, dy \le \Pp^x(|X_{\tau_{D}} - z| \in W) \le C_2 \int_W \varphi_D^x(y) \, dy.
\end{equation}

By Proposition \ref{harmonic_measure_boundary} we have $\Pp^x(|X_{\tau_{D}} - z| \in \{r\}) = 0$ so (\ref{kappa_est}) holds for $W = \{r\}$. For any $r_1 > r$ by (\ref{Levy_system_3}) we have $\Pp^x(|X_{\tau_{D}} - z| \in \{r_1\}) = 0$ so (\ref{kappa_est}) holds for $W = \{r_1\}$.
By Proposition \ref{main_thm_for_small_rings} we get that for any $\eps \in (0,r/4]$ and $\eta \in (0,\eps]$  (\ref{kappa_est}) holds for $W = (r+\eps, r + \eps + \eta)$.
It follows that (\ref{kappa_est}) holds for any $W = [a,b]$, where $[a,b] \subset [r,5r/4]$.

By Lemma \ref{I-W} (\ref{kappa_est}) holds for any $W = [a,b]$ such that $[a,b] \subset (5r/4,2r]$ or $[a,b] \subset [nr,(n+1)r]$ (where $n = 2, 3, \ldots$). 

Hence (\ref{kappa_est}) holds for any $W = [a,b] \subset [r,\infty)$.

By a standard monotone class argument, \eqref{kappa_est} holds for all Borel sets $W\subset [r,\infty)$.
In particular, $\Pp^x(|X_{\tau_{D}} - z| \in \cdot)$ is absolutely continuous with respect to Lebesgue measure on $[r,\infty)$. Hence by the Radon-Nikodym theorem there exists a Borel density $f_D^x:[r,\infty) \to [0,\infty)$ such that for any Borel set $W \subset [r,\infty)$ we have
\begin{equation}
\label{existence_density}
\Pp^x(|X_{\tau_{D}} - z| \in W) = \int_W f_D^x(y) \, dy.
\end{equation}
Moreover, \eqref{kappa_est} yields 
\begin{equation*}
C_1 \varphi_D^x(y) \le f_D^x(y) \le C_2 \varphi_D^x(y)
\end{equation*}
for a.e.\ $y\in[r,\infty)$. Redefining $f_D^x$ on a null set, we may assume that the above inequalities hold for all $y\in[r,\infty)$.

We have 
$$
u(x) = \Ee^x(g(X_{\tau_{D}})).
$$
Let $\kappa_D^x$ be the measure defined by (\ref{kappa_definition}). Since $g$ is radial with respect to $z$ we may write $g(X_{\tau_D})=\tilde g(|X_{\tau_D}-z|)$, hence  $u(x) = \int_r^{\infty} \tilde{g}(y) \, d\kappa_D^x(y)$. Combining this with \eqref{existence_density}, we obtain $u(x) = \int_r^{\infty} \tilde{g}(y) f_D^x(y)\, dy$. 
\end{proof}

\section{Comparison inequalities for $\calL_{e_d}$}

The goal of this section is to prove Proposition \ref{f_b_estimate1}.
We reduce the proof to the following two propositions. 

\begin{proposition}
\label{f_b_estimate}
Set $z = 0$, fix $r > 0$, $\eps \in (0,r/4]$ and $\eta \in (0,\eps]$. There exists a function $\theta \in \calG(r,\eps)$ (depending only on $r$, $\eps$ and $\alpha$)  and there exists a constant $b_1 > 0$ (depending only on $\alpha$) such that
$$
b_1 \calL_{e_d} f_{\theta}(x) \le - \frac{ \eps^{\alpha/2}}{\eta r^{1-\alpha/2} } \calL_{e_d} g(x)\quad \text{for any $x \in B(0,r)$.}
$$
\end{proposition}

\begin{proposition}
\label{F_b_estimate}
Set $z = 0$, fix $r > 0$, $\eps \in (0,r/4]$ and $\eta \in (0,\eps]$. There exists a constant $b_2 > 0$ (depending only on $\alpha$) such that
$$
b_2 \calL_{e_d} F_{\Theta}(x) \ge -\frac{ \eps^{\alpha/2}}{\eta r^{1-\alpha/2} } \calL_{e_d} g(x)\quad \text{for any $x \in B(0,r)$.}
$$
\end{proposition}

Assuming Propositions \ref{f_b_estimate},  \ref{F_b_estimate} we obtain Proposition \ref{f_b_estimate1}.

\begin{proof}[proof of Proposition \ref{f_b_estimate1}]
By Proposition \ref{f_b_estimate} there exist $\theta \in \calG(r,\eps)$ and $b_1 > 0$ such that $\calL_{e_d} f_{b_1,\theta}(x) \le 0$ for any $x \in B(0,r)$, where $f_{b,\theta}$ is defined in (\ref{f_b_definition}). By Lemma \ref{lemma_L_e_d} $\calL f_{b_1,\theta}(x) \le 0$ for any $x \in B(0,r)$. Similarly, by Proposition \ref{F_b_estimate} there exists $b_2 > 0$ such that $\calL_{e_d} F_{b_2,\Theta}(x) \ge 0$ for any $x \in B(0,r)$, where $F_{b,\Theta}$ is defined in (\ref{F_b_definition}). By Lemma \ref{lemma_L_e_d} 
$\calL F_{b_2,\Theta}(x) \ge 0$ for any $x \in B(0,r)$.
\end{proof}

The remainder of this section is devoted to the proofs of Propositions \ref{f_b_estimate},  \ref{F_b_estimate}. 

Unless stated otherwise, we assume that all constants in this section are positive and may depend only on $\alpha$. In this section we write $f_1(x) \approx f_2(x)$ for $x \in U$ when there exist constants $c_1, c_2 > 0$ depending only on $\alpha$ such that for all $x \in U$ we have $c_1 f_1(x) \le f_2(x) \le c_2 f_1(x)$. Similarly, we write $f_1(x) \lesssim f_2(x)$ for $x \in U$ when there exists a constant $c > 0$ depending only on $\alpha$ such that for all $x \in U$ we have $f_1(x) \le c f_2(x)$. 

Set $z = 0$. Fix $r > 0$, $\eps \in (0,r/4]$ and $\eta \in (0,\eps]$. 
For any $\theta \in \calG(r,\eps)$, $\Theta$ given by (\ref{Theta}) and $y \in \R^d$ denote
\begin{equation*}
\varphi(y) =  \left\{              
\begin{array}{lll}                   
\theta(|y|^2),&\text{for}& y \in B(0,r), \\  
0 &\text{for}& y \in B^c(0,r).
\end{array}
\right. 
\end{equation*}

\begin{equation*}
\varPhi(y) =  \left\{              
\begin{array}{lll}                   
\Theta(|y|^2),&\text{for}& y \in B(0,r), \\  
0 &\text{for}& y \in B^c(0,r).
\end{array}
\right. 
\end{equation*}

For any $x = (x_1, \ldots, x_d) \in \R^d$ put $\tilde{x} = (x_1, \ldots, x_{d-1},0) \in \R^{d}$. Without loss of generality we may assume that $x_d \ge 0$. 
For any $x \in B(0,r)$ we define $S_1(x)$, $S_2(x)$, $S_3(x)$, $S_4(x)$  by
\begin{eqnarray*}
S_1(x) &=& \sqrt{[(r - \eps)^2 - |\tilde{x}|^2] \vee 0},\\
S_2(x) &=& \sqrt{r^2 - |\tilde{x}|^2},\\
S_3(x) &=& \sqrt{(r + \eps)^2 - |\tilde{x}|^2},\\
S_4(x) &=& \sqrt{(r + \eps + \eta)^2 - |\tilde{x}|^2}.
\end{eqnarray*}
Similarly, for any $x \in B(0,r)$ and $N \ge 4$ we define $S_*(x)$, $S_{**}(x)$, $S_{***}(x)$ by
\begin{eqnarray*}
S_{*}(x) &=& \sqrt{[(r - \eps + \eps/N)^2 - |\tilde{x}|^2] \vee 0},\\
S_{**}(x) &=  & \sqrt{[(r - 3\eps/4)^2 - |\tilde{x}|^2] \vee 0},\\
S_{***}(x) &=& \sqrt{[(r - \eps/2)^2 - |\tilde{x}|^2] \vee 0}.
\end{eqnarray*} 
It is clear that we have the following relations 
$$S_1(x)\le  S_{*}(x) \le S_{**}(x) \le S_{***}(x) \le S_2(x)\le S_3(x)\le S_4(x).$$

%Note that for $x \in B(0, r - \eps)$ we have
%\begin{eqnarray}
%\nonumber
%S_2(x) - S_1(x) &=&
%\sqrt{r^2 - |\tilde{x}|^2} - \sqrt{(r - \eps)^2 - |\tilde{x}|^2}\\
%\nonumber
%&=& \frac{\eps (2r -\eps)}{\sqrt{r^2 - |\tilde{x}|^2} + \sqrt{(r - \eps)^2 - |\tilde{x}|^2}}\\
%\nonumber
%&\le& \frac{2 \eps r}{\sqrt{r^2 - |\tilde{x}|^2}}\\
%\label{S_2_S_1}
%&\le& \frac{2 \eps \sqrt{r}}{\sqrt{r - |\tilde{x}|}}
%\end{eqnarray}
Recall that $q = r^2 - (r-\eps)^2$.
For  $x \in B(0, r)$, we have the following estimate
\begin{equation} \label{S_2_S_1N}
S_2(x) - S_1(x)\le \frac{q}{S_2(x) }.
\end{equation}
Indeed, for $|\tilde{x} |\le r-\eps$ we have $S^2_2(x) - S^2_1(x)=q$, hence
\begin{equation*}
S_2(x) - S_1(x) = \frac{S^2_2(x) - S^2_1(x)}{S_2(x) + S_1(x)}
= \frac{q}{S_2(x) + S_1(x)}
\le \frac{q}{S_2(x) }
\end{equation*}
and if $|\tilde{x} |\ge r-\eps$ then $S_1(x)=0$ and  $S^2_2(x)\le q$, which proves  (\ref{S_2_S_1N}) in this case.

Similarly, for $x \in B(0, r)$ we obtain
\begin{eqnarray}
\nonumber
S_4(x) - S_3(x) &=&
\sqrt{(r + \eps + \eta)^2 - |\tilde{x}|^2} - \sqrt{(r + \eps)^2 - |\tilde{x}|^2}\\
\nonumber
&=& \frac{\eta (2r+2\eps+\eta)}{\sqrt{(r + \eps + \eta)^2 - |\tilde{x}|^2} + \sqrt{(r + \eps)^2 - |\tilde{x}|^2}}\\
\label{S_4_S_3}
&\approx& \frac{\eta \sqrt{r}}{\sqrt{r + \eps - |\tilde{x}|}}\approx \frac{\eta r}{S_3(x)}.
\end{eqnarray}
\begin{equation}
\label{S_3_x_d}
S_3(x) - x_d =
\sqrt{(r + \eps)^2 - |\tilde{x}|^2} - x_d
= \frac{(r + \eps - |x|)(r + \eps + |x|) }{\sqrt{(r + \eps )^2 - |\tilde{x}|^2} + x_d}.
\end{equation}

Now we present several estimates involving the functions $\lambda, \varphi$ and $\varPhi$, which will be used in the sequel. Recall that we assume $x_d \ge 0$.  

Let $0< a\le b$ then, by concavity of the power function $(0, \infty)\ni v\to v^{\alpha/2}$, we have 
$$ b^{\alpha/2}-a^{\alpha/2}\le (\alpha/2)a^{\alpha/2-1}(b-a).$$
%%%%
Applying the above elementary inequality and assuming $x_d+ |t |< S_2(x)$  we have 
%%%%%
\begin{eqnarray}    
\left| \lambda(x + te_d ) - \lambda(x - t e_d) \right| &=& |(S^2_2(x)-(x_d+t)^2))^{\alpha/2}-  (S^2_2(x)-(x_d-t)^2)^{\alpha/2}|\nonumber\\
&\le&  \frac\alpha 2 (S^2_2(x)-(|x_d|+|t|)^2)^{\alpha/2-1}|(x_d+t)^2 -(x_d-t)^2|\nonumber \\ 
&=& 2\alpha (S^2_2(x)-(|x_d|+|t|)^2)^{\alpha/2-1}|t||x_d|\nonumber\\ \label{lambda}
&\le& 2\alpha (S^2_2(x)-(|x_d|+|t|)^2)^{\alpha/2-1}|t|S_2(x). 
\end{eqnarray}

Now we turn to the estimates  involving the function $\varphi$. We assume that $\theta \in \calG(r,\eps)$.
When $x_d+ |t |< S_2(x)$ we have
\begin{equation}
\label{estimates_double_2a}
\frac{d}{dt} \varphi(x + t e_d) = \frac{d}{dt} \theta(|x + t e_d|^2)
= \theta'(|\tilde{x}|^2 + (t + x_d)^2) 2 (t + x_d)
\end{equation}
and
\begin{equation}
\label{estimates_double_3}
\frac{d^2}{dt^2} \varphi(x + t e_d) = 
 \theta''(|\tilde{x}|^2 + (t + x_d)^2) 4 (t + x_d)^2 + 2 \theta'(|\tilde{x}|^2 + (t + x_d)^2).
\end{equation}

Hence, if $x_d + |t|\le S_2(x)$, by (\ref{estimates_double_2a}) and Lemma \ref{theta_estimates},  we get 
$$
\left| \frac{d}{dt} \varphi(x + t e_d) \right| 
\le \frac{3}{q^2} 2 |x_d + t| 
\le \frac{6 S_2(x)}{q^2}.
$$
So, for some  $|\xi| \le S_2(x)-x_d$, we arrive at
\begin{equation}
\left| \varphi(x + t e_d) - \varphi(x) \right| =
|t| \left| \left[\frac{d}{dt}\varphi(x + t e_d)\right]_{t = \xi} \right|
\le 6 |t|\frac{ S_2(x)}{q^2}\label{varphi}.
\end{equation}

Similarly,
using the above estimates, condition (iv) in Definition \ref{Class_A_1}, Lemma \ref{theta_estimates} and (\ref{estimates_double_3}) we get 
$$
\frac{d^2}{dt^2} \varphi(x + t e_d) \le \frac{2}{q^3} 4 (x_d + t)^2 + \frac{6}{q^2} 
\le  8 \frac{S_2^2(x)} {q^3} + \frac{6}{q^2}.$$
%%%%%%
By the above estimate, for some $|\xi| \le S_2(x)-x_d$, we obtain
\begin{equation}
\label{estimates_double_4}
\varphi(x + t e_d) + \varphi(x - t e_d) - 2 \varphi(x) = t^2 \left[\frac{d^2}{dt^2}\varphi(x + t e_d)\right]_{t = \xi}
\le    t^2 \left(\frac{8 S_2^2(x)} {q^3} + \frac{6}{q^2}\right).
\end{equation}

 By similar arguments,   for $s \in (-S_2(x),S_2(x))$, we have %due to Lemma \ref{}
\begin{equation}\label{diff1}
\left| \varPhi(\tilde{x} + s e_d) - \varPhi(x) \right| =
|s-x_d| \left| \left[\frac{d}{dt}\varPhi(x + t e_d)\right]_{t = \xi} \right|
\le 6 |s-x_d| \frac{S_2(x)}{q^{2} }.
\end{equation}
%where $\xi \in (s,x_d)$ or $\xi \in (x_d,s)$ .
%\begin{equation}
%\left| \varPhi(x + t e_d) - \varPhi(x) \right| =
%|t| \left| \left[\frac{d}{dt}\varPhi(x + t e_d)\right]_{t = \xi} \right|
%\le 6 |t| \frac{S_2(x)}{q^{2} },\label{diff}
%\end{equation}

%Note  that  for $s \in (-S_2,S_2)$ due to Lemma \ref{}
%\begin{equation}\label{diff1}
%\left| \varPhi(\tilde{x} + s e_d) - \varPhi(x) \right| =
%|s-x_d| \left| \left[\frac{d}{dt}\varPhi(x + t e_d)\right]_{t = \xi} \right|
%\le 6 |s-x_d| \frac{S_2(x)}{q^{2} },
%\end{equation}
%where $\xi \in (s,x_d)$ or $\xi \in (x_d,s)$ . 

Next, we state several auxiliary lemmas that will be used to derive estimates for $\calL_{e_d} f_{\theta}$ and $\calL_{e_d} F_{\Theta}$.
\begin{lemma}
\label{under_S_1_estimates}
Assume that $|x| \in [r - \eps, r)$, $|\tilde{x}| \le r - 2 \eps$, $x_d \ge 0$ and $\theta \in \calG(r,\eps)$. Let $S_1= S_1(x)$ and $S_2= S_2(x)$. Then we have
$S_1 \ge \frac 12(S_2-S_1)\ge \frac 13(S_2-S_1)\ge \frac14\eps$
and there exists a constant $\kappa_1$ such that
$$
\int_{S_1 -\frac 12(S_2-S_1)}^{S_1 -\frac 13(S_2-S_1)}
\frac{\lambda(\tilde{x} + s e_d) (\varphi(\tilde{x} + S_1 e_d) - \varphi(\tilde{x} + s e_d))}{|s - x_d|^{1 + \alpha}} \, ds 
\ge \kappa_1 \frac{S_2^{\alpha}}{q^{1 + \alpha/2}}.
$$
\end{lemma}
\begin{proof}
Using $ |\tilde{x}| \le r- 2 \eps$ we get 
$$\frac{S^2_1}{S^2_2}=\frac{(r-\eps)^2- |\tilde{x}|^2}{r^2- |\tilde{x}|^2}
\ge \frac{(r-\eps)^2- (r-2\eps)^2}{r^2- (r-2\eps)^2}= \frac{(2r-3\eps)}{2 (2r-2\eps)}\ge \frac{5}{12},$$
since $r\ge 4\eps$.
 This implies that 
%%%%
$$S_1 \ge \frac 12(S_2-S_1)\ge \frac q{4S_2}\ge \frac q{4r}= \frac{(2r-\eps)\eps}{4r}\ge \frac38\eps .$$
For $s \in (0,S_1)$ we have
\begin{eqnarray*}
\varphi(\tilde{x} + S_1 e_d) - \varphi(\tilde{x} + s e_d)
&=& \frac{1}{r^2 - |\tilde{x}|^2 - S_1^2} - \frac{1}{r^2 - |\tilde{x}|^2 - s^2}\\
&=& \frac{S^2_1 - s^2}{q (S^2_2 - s^2)}.
\end{eqnarray*}
Hence for $s \in (0,S_1)$ we have
\begin{equation}
\label{under_S_1_estimates_1}
\lambda(\tilde{x} + s e_d) (\varphi(\tilde{x} + S_1 e_d) - \varphi(\tilde{x} + s e_d)) = 
\frac{ S^2_1 - s^2}{q  (S_2^2- s^2)^{1 - \alpha/2}}.
\end{equation}
Since $S_1 \ge \frac 12(S_2-S_1)$, for $s \in \left(S_1 - \frac 12(S_2-S_1), S_1\right)$ we have
\begin{eqnarray*}
S_2^2 - s^2 
&\le& S_2^2 - \left(S_1 -\frac 12(S_2-S_1)\right)^2\\
&=&  S_2^2 - S_1^2 - \left(\frac 12(S_2-S_1)\right)^2 +  S_1(S_2-S_1) \\
&\le&  S_2^2 - S_1^2  +  S_1(S_2-S_1) \\
&\le& 2q .
\end{eqnarray*}
Therefore, for $s \in \left(S_1 - \frac 12(S_2-S_1), S_1\right)$ we have
\begin{equation}
\label{under_S_1_estimates_2}
\lambda(\tilde{x} + s e_d) (\varphi(\tilde{x} + S_1 e_d) - \varphi(\tilde{x} + s e_d)) \ge 
\frac{S^2_1 - s^2}{q (2q)^{1 - \alpha/2}}.
\end{equation}
By (\ref{S_2_S_1N}), for $s \in \left(S_1- \frac 12(S_2-S_1), S_1\right)$, we get 
\begin{eqnarray}
|s - x_d| &\le & S_2-\left(S_1- \frac 12(S_2-S_1)\right)
 = \frac 32(S_2-S_1)
\le \frac 32 \frac q{S_2}. \label{under_S_1_estimates_3}
\end{eqnarray}
By (\ref{under_S_1_estimates_2}) and (\ref{under_S_1_estimates_3}) we get
\begin{eqnarray}
&&
\int_{S_1 -\frac 12(S_2-S_1)}^{S_1 -\frac 13(S_2-S_1)}
\frac{\lambda(\tilde{x} + s e_d) (\varphi(\tilde{x} + S_1 e_d) - \varphi(\tilde{x} + s e_d))}{|s - x_d|^{1 + \alpha}} \, ds
\nonumber\\
\nonumber%\label{under_S_1_estimates_4} 
&\ge&
\left(\frac {3q}{2S_2}\right)^{-1-\alpha}\frac{S_1}{ q (2q)^{1 - \alpha/2}}
\int_{S_1 -\frac 12(S_2-S_1)}^{S_1 -\frac 13(S_2-S_1)}
(S_1 - s) \, ds\\
&=&\nonumber
   \frac{8^{\alpha/2}}{3^{1+\alpha}}S_2^{\alpha}\frac{S_1S_2}{q^{3 + \alpha/2}}\left( \frac 12\left(\frac 14-\frac 19\right)\right)(S_2-S_1)^2\\
	&=&\nonumber
  \frac{8^{\alpha/2}}{3^{1+\alpha}}\frac5{72}S_2^{\alpha}\frac{S_1S_2}{q^{3 + \alpha/2}} \frac {q^2} {(S_2+S_1)^2}\\
&\ge&\nonumber
  \frac{8^{\alpha/2-1}}{3^{3+\alpha}} \frac{S_2^{\alpha}}{q^{1 + \alpha/2}},\end{eqnarray}
since $S_1/S_2\ge \sqrt{5/12}$ implies $S_2S_1 (S_2+S_1)^{-2} \ge 1/5$.
\end{proof}

\begin{lemma}
\label{estimates_double}
Assume that $|x| \in [r - \eps, r - 3\eps/4]$, $x_d \ge 0$, $\theta \in \calG(r,\eps)$. Let $S_1= S_1(x)$ and $S_2= S_2(x)$. Then there exists a constant $\kappa_2$ such that for any $\rho \in (0,1/4]$ we have
\begin{eqnarray*}
&& \int_{|t| \le \rho \eps}
\frac{\left(\lambda(x + t e_d) (\varphi(x + t e_d) - \varphi(x)) 
+ \lambda(x - t e_d) (\varphi(x - t e_d) - \varphi(x))\right)}{|t|^{1+\alpha}} \, dt\\ 
&\le& \kappa_2 \rho^{2 - \alpha} \frac{S_2^\alpha}{q^{1+\alpha/2}}.
\end{eqnarray*}
\end{lemma}
\begin{proof} In the whole proof we assume that $\rho \in (0,1/4]$ and $t \in (-\rho \eps, \rho \eps)$. Since $|x|\le r-3\eps/4$ we have $x_d \le S_{**}= S_{**}(x)=\sqrt{(r-3\eps/4)^2-|\tilde{x}|^2}$.  
 Hence for $|t|\le \rho\eps\le \eps/4$ we obtain 
$$x_d + |t| \le S_{**}+ \eps/4  = S_2+ S_{**}-S_2 + \eps/4 < S_2,$$
since $$S_2-S_{**}= \frac  {S^2_2-S^2_{**}}{S_2+S_{**}}\ge \frac{ r^2 -(r-3\eps/4)^2}{2r}>\eps/4.$$
%%%%%%%%%%%%%%%%%%%%%%%%%%%%%%%
Moreover,  we have
\begin{equation} \frac43 S_2^2=\frac43 (r^2-|\tilde{x}|^2)\ge \frac43\left (r^2-(r-3\eps/4)^2\right)\ge q. \label{qest}\end{equation}

We need to properly estimate 
\begin{eqnarray}
\label{estimates_double_0}
&& \left(\lambda(x + t e_d) (\varphi(x + t e_d) - \varphi(x)) 
+ \lambda(x - t e_d) (\varphi(x - t e_d) - \varphi(x))\right)\\
\label{estimates_double_1}
&=& \lambda(x - t e_d) \left(\varphi(x + t e_d) + \varphi(x - t e_d) - 2 \varphi(x)\right)\\
\label{estimates_double_2}
&+& \left(\lambda(x + t e_d) - \lambda(x - t e_d)  \right) \left(\varphi(x + t e_d) - \varphi(x) \right).
\end{eqnarray}
First we will consider (\ref{estimates_double_1}). Applying \eqref{estimates_double_4} we have, by \eqref{qest},
\begin{equation}
\label{estimates_double_4n}
\varphi(x + t e_d) + \varphi(x - t e_d) - 2 \varphi(x)  \le 16 \frac{S_2^2} {q^3}t^2.
\end{equation}

We also have
$$
|x + t e_d|^2 \ge |x|^2 - 2 |t| x_d \ge (r - \eps)^2 - 2 \rho \eps r \ge (r - 2 \eps)^2,
$$
so
$$
\lambda(x + t e_d) \le \lambda((r - 2 \eps) e_d) \le (4 \eps r)^{\alpha/2} \le 4 (\eps r)^{\alpha/2} \le 4 q^{\alpha/2} .
$$
By this and (\ref{estimates_double_4n}) we obtain
\begin{eqnarray}
\nonumber
\int_{|t| \le \rho \eps}
\frac{\lambda(x - t e_d) \left(\varphi(x + t e_d) + \varphi(x - t e_d) - 2 \varphi(x)\right)}{|t|^{1 + \alpha}} \, dt
&\le& 64 q^{\alpha/2} \frac{S_2^2} {q^3} 
\int_{|t| \le \rho \eps} \frac{t^2 dt}{|t|^{1+\alpha}}\\ 
\nonumber
&=& \frac{128}{2 - \alpha} \rho^{2 - \alpha} \eps^{2 - \alpha}  q^{\alpha/2} \frac{S_2^2} {q^3}\\ 
\label{estimates_double_5}
&\le& \frac{128}{2 - \alpha} \rho^{2 - \alpha}  \frac{S_2^\alpha} {q^{1+\alpha/2}}. 
\end{eqnarray}
Here we used $\eps^{2 - \alpha}  S_2^2= (\eps S_2)^{2 - \alpha} {S_2^{\alpha}}\le
 (r\eps)^{2 - \alpha} {S_2^\alpha}\le q^{2 - \alpha} {S_2^\alpha}$.

Next, we consider (\ref{estimates_double_2}). We have
\begin{eqnarray*}
|\tilde{x}|^2 + (x_d + |t|)^2 
&=& |x|^2 + 2 x_d |t| + t^2\\
&\le& (r - 3 \eps/4)^2 + 2 r \eps/4 + (\eps/4)^2\\
&<& (r - \eps/4)^2.
\end{eqnarray*}
This implies that  $$(S^2_2-(x_d+|t|)^2\ge r^2 - (r - \eps/4)^2\ge q/4,$$ which together with \eqref{lambda} leads to the following estimate
\begin{eqnarray*}    
\left| \lambda(x + t e_d) - \lambda(x - t e_d) \right| 
&\le& 8\alpha q^{\alpha/2-1}|t|S_2.\end{eqnarray*}
This and  \eqref{varphi} yield
\begin{eqnarray*}
 \left|\lambda(x + t e_d) - \lambda(x - t e_d)  \right| \left|\varphi(x + t e_d) - \varphi(x) \right|
&\le& 48 \alpha  |t|^2  \frac {S_2^2}{q^{3-\alpha/2}}.
\end{eqnarray*}
By the same arguments, which lead to (\ref{estimates_double_5}) we obtain
\begin{equation*}
\int_{|t| \le \rho \eps}
\frac{\left(\lambda(x + t e_d) - \lambda(x - t e_d)  \right) \left(\varphi(x + t e_d) - \varphi(x - t e_d) \right)}{|t|^{1 + \alpha}} \, dt
\le 
\alpha\frac{96}{2 - \alpha} \rho^{2 - \alpha}  \frac{S_2^\alpha} {q^{1+\alpha/2}}, 
\end{equation*}
which together with (\ref{estimates_double_0}-\ref{estimates_double_2}) and (\ref{estimates_double_5}) imply the assertion of the lemma.
\end{proof}

\begin{lemma}
\label{estimates_S_1_S_2_minus_W}
Assume that $|x| \in [r - \eps, r - 3\eps/4]$, $x_d \ge 0$, $\theta \in \calG(r,\eps)$. Let $S_1= S_1(x)$ and $S_2= S_2(x)$. For any $t > 0$ put $W_{x, t} = [x_d - t, x_d + t]$. Then there is a constant $\kappa_3$ such that for any $\rho \in (0,1/4]$ we have
$$
\int_{(S_1,S_2) \setminus W_{x, \rho \eps}} 
\frac{\lambda(\tilde{x} + s e_d) (\varphi(\tilde{x} + s e_d) - \varphi(x)) }{|s - x_d|^{1 + \alpha}} \, ds
\le \frac{\kappa_3}{\rho^{1 + \alpha}} 
\frac{r S^{\alpha}_2 (\|\theta\|_{\infty} - 1/q)}{ \eps^{1 + \alpha}}.
$$
\end{lemma}
\begin{proof}
For $s \in (S_1,S_2)$ we clearly have 
$\lambda(\tilde{x} + s e_d) = (S_2^2-s^2)^{\alpha/2}\le S^{\alpha}_2$ and $\varphi(\tilde{x} + s e_d) - \varphi(x) \le \|\theta\|_{\infty} - 1/q$. For $s \in (S_1,S_2) \setminus W_{x, \rho \eps}$ we also have $|s - x_d|^{-1 - \alpha} \le \rho^{-1 - \alpha} \eps^{-1 - \alpha}$. We finally observe that  for $|x| \in [r - \eps, r - 3\eps/4]$ we have  $q\le \frac43S^2_2\le \frac43S_2r$, see \eqref{qest}. The above estimates and (\ref{S_2_S_1N}) give the assertion of the lemma.
\end{proof}

\begin{lemma}
\label{S_**_N} Let $x \in B(0,r - \eps)$ and let $S_2= S_2(x)$, $S_{**}= S_{**}(x)$. 
We assume that $N \ge 4$ is such that $\frac{1}{N} < \frac{\eps}{r(r \vee 1)^2}$, $\theta \in \calG(r,\eps)$ and $K$ in the definition of $\theta$ satisfies $K < \eps/N$ and $\|\theta\|_{\infty} < \theta((r - \eps)^2) + 1/N^{4+\alpha}$. Under the above assumptions on $\theta$ and $N$  for $s \in (S_{**},S_2)$ we have
$$
\theta(|\tilde{x} + s e_d|^2) = \|\theta\|_{\infty} < \frac{1}{r^2 - |\tilde{x}|^2 - S_{**}^2}.
$$
\end{lemma}
\begin{proof}
Since $r \ge 4 \eps$ we get $(r \vee 1)^2 \ge 16 \eps^2 > 8 \eps^2$, so $1 > 8 \eps^2/(r \vee 1)^2$. This and the assumption $\frac{1}{N} < \frac{\eps}{r(r \vee 1)^2}$ implies
\begin{equation*}
\frac{1}{8 r \eps} > \frac{\eps}{r (r \vee 1)^2} > \frac{1}{N}.
\end{equation*}
We also have
\begin{equation}
\label{r_eps_N}
\frac{1}{r^2 - (r - 3 \eps/4)^2} - \frac{1}{r^2 - (r - \eps)^2} 
=
\frac{\frac{r \eps}{2} - \frac{7 \eps^2}{16}}{\left(2r \eps - \eps^2\right)\left(\frac{3r \eps}{2} - \frac{9 \eps^2}{16} \right)}
\ge \frac{1}{8 r \eps}
> \frac{1}{N}
> \frac{1}{N^{4+\alpha}}.
\end{equation}

Note that for $x \in B(0, r - \eps)$ and $s \in (S_{**},S_2)$ we have 
\begin{equation*}
|\tilde{x} + s e_d|^2 
\ge |\tilde{x}|^2 + S_{**}^2
= \left(r - \eps + \frac{\eps}{4}\right)^2
\ge \left(r - \eps + K\right)^2.
\end{equation*}
It follows that $\theta(|\tilde{x} + s e_d|^2) = \|\theta\|_{\infty}$. By this, the assumption $\|\theta\|_{\infty} < \theta((r - \eps)^2) + 1/N^{4+\alpha}$ and (\ref{r_eps_N}) we obtain
$$
\theta(|\tilde{x} + s e_d|^2) 
< \theta((r - \eps)^2) + 1/N^{4+\alpha}
< \frac{1}{r^2 - (r - 3 \eps/4)^2}
= \frac{1}{r^2 - |\tilde{x}|^2 - S_{**}^2}.
$$
\end{proof}

\begin{lemma}
\label{estimates_double1}
Assume that $|x| \in [r - \eps, r)$, $x_d \ge 0$. Let $S_1= S_1(x)$, $S_2= S_2(x)$ and $p(x)=S_2-x_d$. Then there exists a constant $\kappa_4$ such that 
\begin{eqnarray*}
\int_{|t| \le p(x)}
\frac{|\lambda(x + t e_d) - \lambda(x - t e_d)| |\varPhi(x + t e_d) - \varPhi(x) |}{|t|^{1+\alpha}} \, dt 
&\le& \kappa_4 \frac {S^\alpha_2}{q^{1+\alpha/2}}.
\end{eqnarray*}
\end{lemma}
\begin{proof} 
 For  $|t| \le p(x)$ we have $x_d+|t|\le S_2$ hence, by \eqref{diff1}, 
 we get
\begin{equation}
\label{PhitSq}
\left| \varPhi(x + t e_d) - \varPhi(x) \right| 
\le 6 |t| \frac{S_2}{q^{2} }.
\end{equation}
If $|t|\le p(x)/2$ then $S^2_2-(x_d+|t|)^2\ge \frac12( S^2_2-x_d^2)$ hence by (\ref{lambda}) we get
$$\left| \lambda(x + t e_d) - \lambda(x - t e_d) \right| \le 
c (S^2_2-x_d^2)^{\alpha/2-1}|t|S_2.$$
Combining the above estimates we get
\begin{eqnarray}
&&\int_{|t| \le p(x)/2}
\frac{|\lambda(x + t e_d) - \lambda(x - t e_d)| |\varPhi(x + t e_d) - \varPhi(x) |}{|t|^{1+\alpha}} \, dt \nonumber \\
&\le& c (S^2_2-x_d^2)^{\alpha/2-1}\frac {S^2_2}{q^{2}}\int_{|t| \le p(x)/2}|t|^{1-\alpha}\, dt\nonumber\\.
&\le& c (S^2_2-x_d^2)^{\alpha/2-1}\frac {S^2_2}{q^{2}}(S^2_2-x^2_d)^{2-\alpha}S_2^{\alpha-2}\nonumber\\
&\le& c (S^2_2-S_1^2)^{1-\alpha/2}\frac {S^\alpha_2}{q^{2}}
\le c\frac {S^\alpha_2}{q^{1+\alpha/2}}.\label{lambda_Phi}
\end{eqnarray}

Moreover, by (\ref{lambda}) and (\ref{PhitSq}) we obtain
\begin{eqnarray}
&&\int_{ p(x)/2\le |t| \le p(x)}
\frac{|\lambda(x + t e_d) - \lambda(x - t e_d)| |\varPhi(x + t e_d) - \varPhi(x) |}{|t|^{1+\alpha}}  \, dt \nonumber\\
&\le& c \frac {S^2_2}{q^{2}}(p(x))^{1-\alpha}S_2^{\alpha/2-1}\int_{ p(x)/2\le t \le p(x)}  (S_2-(x_d+t))^{\alpha/2-1}\, dt \nonumber\\
&\le& c\frac {S^\alpha_2}{q^{1+\alpha/2}}. \label{lambda_PhiN}
\end{eqnarray}
%%%%%%
In the above string of inequalities we used that $(S_2-(x_d+p(x)/2))= p(x)/2$ and  $p(x)\le S_2-S_1\le \frac q{S_2}$. Combining 
\eqref{lambda_Phi} and \eqref{lambda_PhiN} completes the proof. 
\end{proof}
\begin{lemma}\label{Case2L}
Assume that $|x| \in [r - \eps, r)$, $x_d \ge 0$. Let $S_1= S_1(x)$, $S_2= S_2(x)$,
 $p(x)=S_2-x_d$ and $W_x = (-p(x)+x_d,p(x)+x_d) $. There is a constant $\kappa_5$ such that
%%%%%%%%%%%%%
$$ 
\int_{(-S_2, S_2) \setminus W_x}
\frac{\lambda(\tilde{x} + s e_d) |\varPhi(\tilde{x} + s e_d) - \varPhi(x)|}{|s - x_d|^{1 + \alpha}} \, ds 
\le \kappa_5 \frac{S^{\alpha}_2\vee q^{\alpha/2}}{q^{1+\alpha/2} }.
$$  
\end{lemma}
\begin{proof}
For $s \in (-S_2,S_2)$ by Lemma \ref{theta_estimates} we get
\begin{equation*} \label{diff3}
\left| \varPhi(\tilde{x} + s e_d) - \varPhi(x) \right| 
\le  \frac{4}{q }.
\end{equation*}
Next, we observe that for $s\in (-S_2,S_2)\setminus W_x= (-S_2, 2x_d-S_2)$ we have 
$$x_d-s\le S_2-s\le 2(x_d-s). $$
Hence, for $s\in(-S_2,S_2)\setminus W_x$, it follows from \eqref{diff1} that
\begin{eqnarray}
\lambda(\tilde{x} + s e_d) |\varPhi(\tilde{x} + s e_d) - \varPhi(x)|&\le& 6 |s-x_d| \frac{S_2}{q^{2} }(S^2_2-s^2)^{\alpha/2}\nonumber\\
&\le& c \frac{S^{1+\alpha/2}_2}{q^{2} }(x_d-s)^{1+\alpha/2}\label{lambdaPhi}\end{eqnarray}
and
\begin{eqnarray}
\lambda(\tilde{x} + s e_d) |\varPhi(\tilde{x} + s e_d) - \varPhi(x)|&\le&\frac{4}{q }(S^2_2-s^2)^{\alpha/2}\nonumber\\
&\le&c \frac{S^{\alpha/2}_2}{q }(x_d-s)^{\alpha/2}.\label{lambdaPhi1}
\end{eqnarray}
Now we make the assumption $S_2\ge 2x_d$. Then, by \eqref{lambdaPhi1}, we obtain
\begin{eqnarray*}
\int_{-S_2}^{2x_d-S_2} 
\frac{\lambda(\tilde{x} + s e_d) |\varPhi(\tilde{x} + s e_d) - \varPhi(x)|}{|s - x_d|^{1 + \alpha}} \, ds 
&\le&  c\frac{S^{\alpha}_2}{q }\int_{S_2- 2x_d}^{S_2} (s+x_d)^{-\alpha-1}\, ds\\
&\le& c\frac{S^{\alpha}_2}{q }\frac {2x_d} {(S_2- x_d)^{1+\alpha}} \\
&\le& c\frac{1}{q } 
\le c\frac{S^{\alpha}_2\vee q^{\alpha/2}}{q^{1+\alpha/2} }, 
\end{eqnarray*}
which completes the proof in the case $S_2\ge 2x_d$. 

Now we assume that  $S_2\le 2x_d$. 
In this case,  by \eqref{lambdaPhi1}, we observe that  
\begin{eqnarray}
\int_{-S_2}^{0} 
\frac{\lambda(\tilde{x} + s e_d) |\varPhi(\tilde{x} + s e_d) - \varPhi(x)|}{|s - x_d|^{1 + \alpha}} \, ds 
&\le&  c\frac{S^{\alpha}_2}{q }\int_{0}^{S_2} (s+x_d)^{-\alpha-1}\, ds\nonumber\\
&\le& c\frac{S^{\alpha}_2}{q }\frac {S_2} {x_d^{1+\alpha}} \nonumber\\
&\le& c\frac{1}{q } 
\le c\frac{S^{\alpha}_2\vee q^{\alpha/2}}{q^{1+\alpha/2} }.\label{est-S_2}
\end{eqnarray}
Next, by \eqref{lambdaPhi}, we infer that
\begin{eqnarray}
\int_{(S_1, S_2) \setminus W_x}
\frac{\lambda(\tilde{x} + s e_d) |\varPhi(\tilde{x} + s e_d) - \varPhi(x)|}{|s - x_d|^{1 + \alpha}} \, ds 
&\le&  c\frac{S^{1+\alpha/2}_2}{q^{2} } \int_{S_1}^{ x_d} (x_d - s)^{-\alpha/2}\, ds \nonumber\\
&\le& c\frac{S^{\alpha}_2}{q^{1+\alpha/2} }.\label{est_S_1toS_2}, 
\end{eqnarray} 
which is the consequence of the estimate
$$\int_{S_1}^{ x_d} (x_d - s)^{-\alpha/2}\, ds= c(x_d-S_1)^{1-\alpha/2}\le c\frac{(S^2_2-S^2_1)^{1-\alpha/2}}{S_2^{1-\alpha/2}}\le c \left(\frac{q}{S_2}\right)^{1-\alpha/2}.$$ 
In the case  $S_1 > 0$ and $ x_d\ge \frac{S_1+S_2}2$, it follows from \eqref{lambdaPhi1} that
\begin{eqnarray}
\int_{0}^{S_1} 
\frac{\lambda(\tilde{x} + s e_d) |\varPhi(\tilde{x} + s e_d) - \varPhi(x)|}{|s - x_d|^{1 + \alpha}} \, ds 
&\le&  c\frac{S^{\alpha/2}_2}{q }\int_{0}^{ S_1} (x_d - s)^{-\alpha/2-1}\, ds \nonumber\\
&\le& c\frac{S^{\alpha}_2}{q } (S^2_2-S^2_1)^{-\alpha/2}\nonumber\\
&\le& c\frac{S^{\alpha}_2}{q^{1+\alpha/2} }. \label{est_0toS_1}
\end{eqnarray} 
In the last step we used the fact that $S^2_2-S^2_1=q$ if $S_1>0$. 
%%%%%%%%%%%%%%%%%%%%%
Combining  the estimates \eqref{est_S_1toS_2} and \eqref{est_0toS_1}  we have in the case $x_d\ge \frac{S_1+S_2}2$ that
\begin{eqnarray}
\int_{(0, S_2) \setminus W_x}
\frac{\lambda(\tilde{x} + s e_d) |\varPhi(\tilde{x} + s e_d) - \varPhi(x)|}{|s - x_d|^{1 + \alpha}} \, ds 
&\le&  c\frac{S^{\alpha}_2}{q^{1+\alpha/2} }.\label{0toS_2W}
\end{eqnarray}

Now we consider  the case $S_2/2\le x_d< \frac{S_1+S_2}2$.   Then by \eqref{lambdaPhi1}    we have
\begin{eqnarray*}
\int_{(0, S_2) \setminus W_x} 
\frac{\lambda(\tilde{x} + s e_d) |\varPhi(\tilde{x} + s e_d) - \varPhi(x)|}{|s - x_d|^{1 + \alpha}} \, ds 
&\le&  c\frac{S^{\alpha/2}_2}{q }\int_{0}^{ 2x_d-S_2} (x_d - s)^{-\alpha/2-1}\, ds \nonumber\\
&=& c\frac{S^{\alpha/2}_2}{q }( (S_2 - x_d)^{-\alpha/2}- x_d^{-\alpha/2}) \nonumber\\
&\le& c\frac{S_2^{\alpha/2}}{q (S_2 - x_d)^{\alpha/2}} \nonumber\\
&\le& c\frac{S^{\alpha}_2}{q^{1+\alpha/2} }. 
\end{eqnarray*} 
since in this case 
$S_1>0$, hence  $S^2_2-S^2_1=q$.
The last inequality combined with \eqref{est-S_2} and \eqref{0toS_2W}  completes the proof of the lemma.
\end{proof}    

We now establish a two-sided estimate for $\calL_{e_d} g$. Recall that $z = 0$ so for $x \in D = B(0,r)$ we have $\delta_D(x) = r - |x|$.
\begin{lemma}
\label{S_3_S_4_estimates}
Let $|x| <r$ and $S_2= S_2(x)$. %$x_d \ge 0$.
Then we have
\begin{equation*}
\frac{\eps^{\alpha/2}}{\eta r^{1 - \alpha/2}} \calL_{e_d} g(x)
 \approx \frac{S_2^\alpha \vee q^{\alpha/2}}{q^{1+ \alpha/2}} \left(\frac\eps{\delta_D(x)}\wedge 1\right)^{1+\alpha}.
\end{equation*}
\end{lemma}
\begin{proof} Throughout the proof $S_3=S_3(x)$ and $S_4=S_4(x)$.
Observe that 
$$\calL_{e_d} g(x)=\calA_{\alpha} \int_{S_3}^{S_4} \left(\frac{1}{|s - x_d|^{1+\alpha}} +  \frac{1}{|s + x_d|^{1+\alpha}}\right)\, ds.$$
We first consider  $|x| <r-\eps$.
Since
$$\int_{S_3}^{S_4}   \frac{1}{|s + x_d|^{1+\alpha}}\, ds\le \int_{S_3}^{S_4} \frac{1}{|s -x_d|^{1+\alpha}}\, ds$$ 
we have
$$
\calL_{e_d} g(x)\approx \int_{S_3 - x_d}^{S_4 - x_d} \frac{1}{|t|^{1+\alpha}} \, dt.
$$
Next, by (\ref{S_3_x_d}), we get
$$
S_3 - x_d \approx \frac{\delta_D(x) \sqrt{r}}{\sqrt{r - |\tilde{x}|}}.
$$
Using this and (\ref{S_4_S_3}) we obtain $S_4 - x_d = S_3 - x_d + S_4 - S_3 \approx S_3 - x_d$.
Hence
\begin{equation}
\label{S_3_S_4_estimates_L}
\calL_{e_d} g(x) \approx
 \frac{S_4 - S_3}{(S_3 - x_d)^{1 + \alpha}}
\approx \frac{\eta (r - |\tilde{x}|)^{\alpha/2}}{r^{\alpha/2} \delta_D^{1 + \alpha}(x)}.
\end{equation}

Next we  consider  $r-\eps\le |x| <r$.
We have $r + \eps - |x| \approx \eps$ and $0 \le x_d \le \sqrt{(r+\eps)^2 - |\tilde{x}|^2}$. Hence by (\ref{S_3_x_d}) we get
$$
S_3 - x_d \approx \frac{\eps \sqrt{r}}{\sqrt{r + \eps - |\tilde{x}|}}.
$$
By (\ref{S_4_S_3}) we also have 
$$
S_4 - S_3 \approx \frac{\eta \sqrt{r}}{\sqrt{r + \eps - |\tilde{x}|}}
\le \frac{\eps \sqrt{r}}{\sqrt{r + \eps - |\tilde{x}|}},
$$
so for $s \in (S_3, S_4)$ we have
$$
|s - x_d| = s - S_3 + S_3 - x_d \approx \frac{\eps \sqrt{r}}{\sqrt{r + \eps - |\tilde{x}|}}.
$$
Hence
\begin{equation}
\label{S_3_S_4_estimates_1}
\calL_{e_d} g(x)\approx  \int_{S_3}^{S_4} \frac{1}{|s - x_d|^{1+\alpha}} \, ds
\approx \frac{\eta (r + \eps - |\tilde{x}|)^{\alpha/2}}{r^{\alpha/2} \eps^{1 + \alpha}}.
\end{equation}
Clearly  \eqref{S_3_S_4_estimates_L} and (\ref{S_3_S_4_estimates_1}) give the assertion of the lemma.
\end{proof}

We are now ready to prove the first of the two main propositions in this section.
\begin{proof}[proof of Proposition \ref{f_b_estimate}]
Recall that $r > 0$, $\eps \in (0,r/4]$ and $\eta \in (0,\eps]$ are fixed.

Let $\kappa_1, \kappa_2, \kappa_3$ be constants from Lemmas \ref{under_S_1_estimates}, \ref{estimates_double} and \ref{estimates_S_1_S_2_minus_W}, respectively. Recall that $\calA_{\alpha}$ is defined in (\ref{operator_L}) and $\tilde{\calA}_{\alpha}$ is defined in (\ref{lambda_formula_generator}).

Fix $N \in \R$ satisfying the following conditions.

\begin{equation}
\label{condition_N_1}
N \ge \frac{r(r \vee 1)^2}{\eps},
\end{equation}

\begin{equation}
\label{condition_N_2}
N \ge \left(\frac{2\kappa_2 + 16 \kappa_3}{\kappa_1}\right)^{1/(2-\alpha)},
\end{equation}

\begin{equation}
\label{condition_N_3}
N \ge \left(\frac{\calA_\alpha (2\kappa_2 + 16 \kappa_3)}{\tilde{\calA}_{\alpha}}\right)^{1/(2-\alpha)}.
\end{equation}
Note that (\ref{condition_N_1}) implies that $N \ge r/\eps \ge 4$.

Let $\theta \in \calG(r,\eps)$ be the function which satisfies conditions (\ref{1N}), (\ref{1N4}) for the above $N$ and which existence is proved in Lemma \ref{properties_theta}.% Let $b$ be an arbitary positive number.

We will consider 3 cases:

{\bf Case 1}. $|x| \in (0, r - \eps)$.

{\bf Case 2}. $|x| \in [r -\eps, r - \eps + \eps/N]$.

{\bf Case 3}. $|x| \in (r - \eps + \eps/N, r)$.

Throughout the proof we skip $x$ in the notation of $S_1(x),  S_{*}(x), S_{**}(x), S_{***}(x), \ S_2(x)$. 
\vskip 2pt
First, we consider {\bf Case 1}.
 For $y \in \R^d$ put
$$
u(y) = f_{\theta}(y) -  h(y),
$$
where $h$ is given by (\ref{h_function}).
 By (\ref{h_function_generator}) for any $y \in B(0,r)$ we have $\calL_{e_d} f_{\theta}(y) = \calL_{e_d} u(y)$.
 Note that for any $y \in B(0,r)$ we have
\begin{equation}
\label{u_formula1}
u(y) =  (r^2 - |y|^2)^{\alpha/2} \left(\varphi(y) - \frac{1}{r^2 - |y|^2}\right).
\end{equation} 
 Note also that $u(y) = 0$ for $y \in B(0,r - \eps)$. In Case 1 we have $x \in B(0,r - \eps)$ so
\begin{eqnarray*}
\calL_{e_d} f_{\theta}(x) &=& \calL_{e_d} u(x)\\
&=& \calA_{\alpha} \int_{S_1}^{S_2} \frac{u(\tilde{x} +s e_d)}{|s - x_d|^{1+\alpha}} \, ds
+   \calA_{\alpha} \int_{-S_2}^{-S_1} \frac{u(\tilde{x} +s e_d)}{|s - x_d|^{1+\alpha}}  \, ds\\
&=& \text{I} + \text{II}.
\end{eqnarray*}
Denote $g(v) = (r^2 - v)^{-1}$. By the fact that $g''(v)$ is increasing on $(0,r^2)$ and (iv) in Definition \ref{Class_A_1}  we obtain  $\varphi(\tilde{x} + s e_d) - \frac{1}{r^2 - |\tilde{x} + s e_d|^2} \le 0$ for $s \in (S_1, S_2)$. Using this, (\ref{u_formula1}) and Lemma \ref{S_**_N} we obtain 
\begin{eqnarray}
\nonumber
\text{I}
&\le& \calA_{\alpha} \int_{S_{**}}^{S_{***}} 
\frac{(r^2 - |\tilde{x} + s e_d|^2)^{\alpha/2}}{|s - x_d|^{1+\alpha} }
\left(\varphi(\tilde{x} + s e_d) - \frac{1}{r^2 - |\tilde{x} + s e_d|^2}\right) \, ds\\
\label{I_estimate_1}
&\le& \calA_{\alpha} \int_{S_{**}}^{S_{***}} 
\frac{ (S_2^2 - s^2)^{\alpha/2}}{|s - x_d|^{1+\alpha} }
\left(\frac{1}{S_2^2 - S_{**}^2} - \frac{1}{S_2^2 - s^2}\right) \, ds. 
\end{eqnarray}
 Since we are in Case 1 for any $s \in (S_{**}, S_{***})$ we have $S_2^2 - s^2 \approx r \eps\approx q$. Hence for any 
$s \in (S_{**}, S_{***})$ we obtain
\begin{equation*}
\frac{1}{S_2^2 - S_{**}^2} - \frac{1}{S_2^2 - s^2 }
= -\frac{s^2 - S_{**}^2}{(S_2^2 - s^2)(r^2 - |\tilde{x}|^2 - S_{**}^2)}
\approx -\frac{s^2 - S_{**}^2}{q^2}.
\end{equation*}
Using this and (\ref{I_estimate_1}) we obtain
\begin{equation}
\label{I_estimate_2}
\text{I} \lesssim \frac{-q^{\alpha/2} }{q^2} \int_{S_{**}}^{S_{***}} \frac{(s - S_{**})(s + S_{**})}{|s - x_d|^{1+\alpha}} \, ds.
\end{equation}
Since we are in Case 1 we clearly have $S_{**} \approx \sqrt{r} \sqrt{r - |\tilde{x}|}\approx S_{2}\approx S_{***}$,
$$
S_{***} - S_{**} = 
\frac{S^2_{***} - S^2_{**}}{S_{***} + S_{**}} 
\approx \frac{r \eps}{S_{2}}\approx \frac{q}{S_{2}}.
$$
For any $s \in (S_{**}, S_{***})$ we also have
%$$
%|s - x_d| = s - x_d \ge S_{**} - x_d \ge \frac{\delta(|x|) \sqrt{r}}{\sqrt{r - |\tilde{x}|}}.
%$$
%POWINNO CHYBA BYC
$$
|s - x_d| = s - x_d \le S_{***} - x_d \le  \frac{\delta_D(x) 2{r}}{S_{***}}\approx \frac{\delta_D(x) {r}}{S_{2}}.
$$
It follows that 
\begin{eqnarray*}
\text{I} &\lesssim& 
{-q^{\alpha/2-2} } 
\frac{ S_{2}}{\left(\frac{\delta_D(x) r}{ S_{2}}\right)^{1 + \alpha}} 
\int_{S_{**}}^{S_{***}} (s - S_{**}) \, ds\\
&=& {-\frac12 q^{\alpha/2-2} } 
\frac{ S_{2}^{2 + \alpha}}{\left(\delta_D(x) r\right)^{1 + \alpha}}(S_{***} - S_{**})^2 \\
&\approx&
- \left (\frac\eps { \delta_D(x)}\right)^{1+\alpha} \frac{  S_{2}^{ \alpha}}{q^{1+\alpha/2}} .
\end{eqnarray*}
We clearly have $\text{II} \le 0$, so
\begin{equation}
\label{Case_1_I_III}
\calL_{e_d} f_{\theta}(x) \lesssim - \left (\frac\eps { \delta_D(x)}\right)^{1+\alpha} \frac{  S_{2}^{ \alpha}}{q^{1+\alpha/2}}.
\end{equation}

\vskip 2pt

Now, we consider {\bf Case 2}. 
We fix $x \in \R^d$ satisfying $|x| \in [r -\eps, r - \eps + \eps/N]$ and for $y \in \R^d$ we define
$$
u(y) = f_{\theta}(y) -\lambda(y) \varphi(x)= \lambda(y) \varphi(y) -\lambda(y) \varphi(x).
$$
Observe that 
$$
\calL_{e_d} f_{\theta}(x) = - \tilde{\calA}_{\alpha}  \varphi(x) 
+ \calL_{e_d} u(x).
$$
Since  
$
u(x) = 0$
we have
\begin{eqnarray*}
&&\calL_{e_d} u(x)\\
&=& 
\frac{\Aa}{2}  \int_{\R \setminus \{0\}} \left[u(x + e_d t) + u(x - e_d t) \right] \, \frac{dt}{|t|^{1 + \alpha}}\\ 
&=& \Aa \int_{|t| \ge \eps/N} u(x + e_d t) \, \frac{dt}{|t|^{1 + \alpha}}
+ \frac{\Aa}{2}  \int_{|t| < \eps/N} \left[u(x + e_d t) + u(x - e_d t) \right] \, \frac{dt}{|t|^{1 + \alpha}}.
\end{eqnarray*}
Put $W_{x, \eps/N} = \{s \in \R: s \in [x_d - \eps/N, x_d + \eps/N]\}$. 
It follows that in Case 2 we have
\begin{eqnarray*}
\calL_{e_d} f_{\theta}(x) &=& - \tilde{\calA}_{\alpha}  \varphi(x)
+ \calA_{\alpha} \int_{(-S_1,S_1) \setminus W_{x, \eps/N}} \frac{u(\tilde{x} +s e_d)}{|s - x_d|^{1+\alpha}} \, ds\\
&+&   \calA_{\alpha} \int_{(S_1,S_2) \setminus W_{x, \eps/N}} \frac{u(\tilde{x} +s e_d)}{|s - x_d|^{1+\alpha}} \, ds\\
&+&   \calA_{\alpha} \int_{(-S_2,-S_1) \setminus W_{x, \eps/N}} \frac{u(\tilde{x} +s e_d)}{|s - x_d|^{1+\alpha}} \, ds\\ 
&+& \frac{\Aa}{2}  \int_{|t| < \eps/N} \left[u(x + e_d t) + u(x - e_d t) \right] \, \frac{dt}{|t|^{1 + \alpha}}\\
&=& \text{I} + \text{II} + \text{III} + \text{IV} + \text{V}.
\end{eqnarray*}
By Lemma \ref{estimates_S_1_S_2_minus_W} and (\ref{1N4}) we obtain
\begin{equation}
\label{Case_2_V_1}
\text{III} =
\calA_{\alpha}  \int_{(S_1,S_2) \setminus W_{x, \eps/N}} 
\frac{\lambda(\tilde{x} + s e_d) (\varphi(\tilde{x} + s e_d) - \varphi(x)) }{|s - x_d|^{1 + \alpha}} \, ds
\le \frac{\kappa_3 \calA_{\alpha} }{N^2} \frac{r}{N} \frac {S^{\alpha}_2 }{ \eps^{1 + \alpha}}.
\end{equation}
%%%%%%%%%%%%%%%
By (\ref{condition_N_1})%and the fact that $|\tilde{x}| < r - 3\eps/4$ in Case 2 we obtain
$$
\frac{r}{N} \le \frac{\eps }{r} \frac1{r}  \le \left(\frac{\eps}{r}\right)^{\alpha/2}\frac1{r}.
$$
Hence (\ref{Case_2_V_1}) is bounded from above by
$$
\frac{\kappa_3 \calA_{\alpha} }{N^2} \frac{S^{\alpha}_2}{(r \eps)^{1 + \alpha/2}}\le 4\frac{\kappa_3 \calA_{\alpha} }{N^2} \frac{S^{\alpha}_2}{q^{1 + \alpha/2}}.
$$
By the same arguments $\text{IV}$ can be estimated from above by the same expression. Hence we obtain
\begin{equation}
\label{Case_2_V_VI}
\text{III} + \text{IV} \le 8\frac{\kappa_3 \calA_{\alpha} }{N^2} \frac{S^{\alpha}_2}{q^{1 + \alpha/2}}.
\end{equation}
By Lemma \ref{estimates_double} we obtain
\begin{eqnarray}
\nonumber
\text{V}
&=&
\frac{\calA_{\alpha} }{2} \int_{|t| \le \eps/N}
\Big( \lambda(x + t e_d) (\varphi(x + t e_d) - \varphi(x)) \\
\nonumber
&& \quad \quad \quad \quad \quad \quad \quad \quad \quad +   \lambda(x - t e_d) (\varphi(x - t e_d) - \varphi(x)) \Big) \frac{dt}{|t|^{1+\alpha}}  
\\
\label{Case_2_VII_1}
&\le& \frac{\calA_{\alpha}  \kappa_2}{2 N^{2 - \alpha}}  \frac{S_2^\alpha}{q^{1+\alpha/2}}.
\end{eqnarray}

Now, let consider two subcases of Case 2:
%%%%%%%%%%
Subcase 2A. $|x| \in [r -\eps, r - \eps + \eps/N]$ and $|\tilde{x}| \le r - 2 \eps$.
%%%%%%%%%%%%%%%%%%%%
Subcase 2B. $|x| \in [r -\eps, r - \eps + \eps/N]$ and $|\tilde{x}| > r - 2 \eps$.

First we consider Subcase 2A. 
Clearly, we have 
\begin{equation}
\label{Case_2A_I}
\text{I} \le 0.
\end{equation} 

Note that $x_d \ge S_1$, so for any $s \in (-S_1,S_1)$ we have
\begin{eqnarray}
\label{Case_2A_II_1}
u(\tilde{x} + s e_d) &=& 
 \lambda(\tilde{x} + s e_d) \left(\varphi(\tilde{x} + s e_d) - \varphi(\tilde{x} + x_d e_d)\right)\\
\label{Case_2A_II_2}
&\le&  \lambda(\tilde{x} + s e_d) \left(\varphi(\tilde{x} + s e_d) - \varphi(\tilde{x} + S_1 e_d)\right) < 0.
\end{eqnarray}
Since $N \ge 4$, $x_d \ge S_1$ and $\frac 13(S_2-S_1)\ge \frac\eps4$ (see Lemma \ref{under_S_1_estimates}) we obtain 
$$
\left[x_d - \frac{\eps}{N}, x_d + \frac{\eps}{N}\right] \cap \left[S_1 - \frac 12(S_2-S_1), S_1 - \frac 13(S_2-S_1)\right] = \emptyset.
$$
Using this (\ref{Case_2A_II_1}), (\ref{Case_2A_II_2}) and Lemma \ref{under_S_1_estimates} we obtain
\begin{eqnarray}
\nonumber
\text{II} &\le&
-\calA_{\alpha}  \int_{S_1 - \frac 13(S_2-S_1)}^{S_1 - \frac 12(S_2-S_1),}
\frac{\lambda(\tilde{x} + s e_d) (\varphi(\tilde{x} + S_1 e_d) - \varphi(\tilde{x} + s e_d))}{|s - x_d|^{1 + \alpha}} \, ds\\
&\le& - \calA_{\alpha} \kappa_1 \frac{S_2^\alpha}{q^{1+\alpha/2}}.
\label{Case_2A_II_3_0}
\end{eqnarray}
Finally, by  (\ref{Case_2_V_VI}), (\ref{Case_2_VII_1}), (\ref{Case_2A_I}) and (\ref{Case_2A_II_3_0}) we obtain
\begin{eqnarray}
\nonumber
\calL_{e_d} f_{\theta}(x)&=&\text{I} +  \text{II} +\text{III}  +\text{IV} + \text{V}\\ 
\nonumber
&\le&
\calA_{\alpha} \frac{S_2^\alpha}{q^{1+\alpha/2}}
\left(- \kappa_1 + \frac{8 \kappa_3}{N^2} + \frac{\kappa_2}{2 N^{2 - \alpha}}\right)\\
\label{Case_2A}
&\le&
\calA_{\alpha} \frac{S_2^\alpha}{q^{1+\alpha/2}}
\left(- \kappa_1 + \frac{8 \kappa_3 + \kappa_2}{N^{2 - \alpha}}\right)
\end{eqnarray}

Now let us consider Subcase 2B. 
In Case 2 we have $|x| \ge r - \eps$, so $\varphi(x) \ge \varphi((r - \eps)^2 e_d) = 1/q$, which gives  
\begin{equation}
\label{Case_2B_I}
\text{I} \le - \tilde{\calA}_{\alpha}  \frac{1}{q} .
\end{equation} 
%%%%%%%%%%
By the same arguments as in (\ref{Case_2A_II_1}-\ref{Case_2A_II_2}) we get 
\begin{equation}
\label{Case_2B_II}
\text{II} \le 0.
\end{equation} 
%%%%%%%%%%%%%%%%
Note that in Subcase 2B we have $r - |\tilde{x}| \le 2 \eps$, so ${S_2^2} \le 2 q$. Using this, (\ref{Case_2B_I}), (\ref{Case_2B_II}), (\ref{Case_2_V_VI}) and (\ref{Case_2_VII_1}) we obtain
\begin{eqnarray}
\nonumber
\calL_{e_d} f_{\theta}(x) &=&  \text{I} + \text{II} + \text{III} + \text{IV} + \text{V}\\
\nonumber
&\le&
\frac{1}{q }
\left(- {\tilde{\calA}_{\alpha} } + 8\frac{S_2^\alpha}{q^{\alpha/2}}\frac{ \kappa_3 \calA_{\alpha} }{N^2} + 
\frac{S_2^\alpha}{q^{\alpha/2}}\frac{\kappa_2 \calA_{\alpha} }{2 N^{2 - \alpha}}\right)\\
&\le& 
\label{Case_2B_I_II_V_VI_VII_2}
\frac1{q}
\left(- \tilde{\calA}_{\alpha} + \frac{8 \kappa_3 \calA_{\alpha} + \kappa_2 \calA_{\alpha}}{N^{2 - \alpha}}
\right).
\end{eqnarray}
%By (\ref{condition_N_3}) the expression in (\ref{Case_2B_I_II_V_VI_VII_2}) is nonpositive, which gives (\ref{Case_2B_I_II_V_VI_VII_1}).
%Next, by the estimate $(r - |\tilde{x}|)^{\alpha/2} \le 2 \eps^{\alpha/2}$, (\ref{Case_2_III_IV}) and (\ref{Case_2B_I}) there exists %(sufficiently big) $b > 0$ (which depends only on $\alpha$) such that $\frac{1}{2} \text{I} + \text{III} + \text{IV} \le 0$.
%%%%%%%%%%
%This and (\ref{Case_2B_I_II_V_VI_VII_1}) imply that  we have in Subcase 2B 
%$$\calL_{e_d} f_{\theta}(x) \le \frac1{q}
%\left(- \frac{\tilde{\calA}_{\alpha}}{2} + \frac{8 \kappa_3 \calA_{\alpha} + \kappa_2 \calA_{\alpha}}{N^{2 - \alpha}}\right).$$

\vskip 2pt
Now, we consider {\bf Case 3}. 
We fix $x \in \R^d$ satisfying $|x| \in (r - \eps + \eps/N, r)$ and for $y \in \R^d$ we define
$$
u(y) = f_{\theta}(y) - \lambda(y) \|\theta\|_{\infty}= f_{\theta}(y) - \lambda(y) \varphi(x)
$$
Note that for $y \in B(0,r) \setminus B(0, r - \eps + \eps/N)$ we have $\varphi(y) = \theta(|y|^2) = \|\theta\|_{\infty}$, so
$$
u(y) = \lambda(y) \varphi(y) - \lambda(y) \|\theta\|_{\infty} = 0.
$$
Hence
$$
\calL_{e_d} f_{\theta}(x) = - \tilde{\calA}_{\alpha}  \|\theta\|_{\infty}
+ \calL_{e_d} u(x)
$$
and
$$
\calL_{e_d} u(x) = 
\calA_{\alpha}  \int_{\R} \frac{u(x + e_d t)}{|t|^{1 + \alpha}} \, dt
= \calA_{\alpha} \int_{\R} \frac{u(\tilde{x} +s e_d)}{|s - x_d|^{1+\alpha}} \, ds. 
$$
It follows that in Case 3 we have
\begin{eqnarray*}
\calL_{e_d} f_{\theta}(x) &=& - \tilde{\calA}_{\alpha}  \|\theta\|_{\infty}
+ \calA_{\alpha} \int_{(-S_*,S_*)} \frac{u(\tilde{x} +s e_d)}{|s - x_d|^{1+\alpha}} \, ds\\
&=& \text{I} + \text{II}.
\end{eqnarray*}

Now, let consider two subcases of Case 3:\newline
Subcase 3A. $|x| \in (r - \eps + \eps/N, r)$ and $|\tilde{x}| \le r - 2 \eps$.\newline
Subcase 3B. $|x| \in (r - \eps + \eps/N, r)$ and $|\tilde{x}| > r - 2 \eps$.

First we consider Subcase 3A. 
Clearly, we have 
\begin{equation}
\label{Case_3A_I}
\text{I} \le 0.
\end{equation} 
Note that for any $s \in (-S_*,S_*)$ we have
\begin{equation}
\label{Case_3A_II_1}
u(\tilde{x} + s e_d) =
 \lambda(\tilde{x} + s e_d) \left(\varphi(\tilde{x} + s e_d) - \|\theta\|_{\infty} \right) \le 0.
\end{equation}
For any $s \in (-S_1,S_1)$ we also have
\begin{equation}
\label{Case_3A_II_2}
u(\tilde{x} + s e_d) \le  \lambda(\tilde{x} + s e_d) \left(\varphi(\tilde{x} + s e_d) - \varphi(\tilde{x} + S_1 e_d)\right) < 0.
\end{equation}
Using  \eqref{Case_3A_I}, (\ref{Case_3A_II_1}), (\ref{Case_3A_II_2}) and Lemma \ref{under_S_1_estimates} we obtain
\begin{eqnarray}
\nonumber
\calL_{e_d} f_{\theta}(x)
&\le &\nonumber
-\calA_{\alpha}  \int_{S_1 - \frac 13(S_2-S_1)}^{S_1 - \frac 12(S_2-S_1),}
\frac{\lambda(\tilde{x} + s e_d) (\varphi(\tilde{x} + S_1 e_d) - \varphi(\tilde{x} + s e_d))}{|s - x_d|^{1 + \alpha}} \, ds\\
&\le& - \calA_{\alpha} \kappa_1 \frac{S_2^\alpha}{q^{1+\alpha/2}}.
\label{Case_2A_II_3}
\end{eqnarray}

Now let us consider Subcase 3B. 
We have $\|\theta\|_\infty > \varphi((r - \eps) e_d) = 1/q $, which gives  
$\text{I} \le - \tilde{\calA}_{\alpha}/q$.
By the same arguments as in (\ref{Case_3A_II_1}) we get 
$\text{II} \le 0$.
Hence in Subcase 3B we have 
\begin{equation}
\label{Case_3B}
\calL_{e_d} f_{\theta}(x) \le - \tilde{\calA}_{\alpha}  \frac{1}{q}.\end{equation} 

Finally, we observe that for $|x|<r $
%%%%
\begin{equation}
\label{shape}
\frac{S_2^\alpha\vee q^{\alpha/2}}{q^{1+\alpha/2}}= \left\{             
\begin{array}{lll}                   
\displaystyle{ {S_2^\alpha}{q^{-1-\alpha/2}}},&\text{for}& |\tilde{x}|< r-\eps, \\  
q^{-1},&\text{for}& |\tilde{x}|\ge r-\eps.
\end{array}       
\right. 
\end{equation}
Next, by (\ref{condition_N_2}), (\ref{condition_N_3}), \eqref{shape} and  the estimates
\eqref{Case_1_I_III},
\eqref{Case_2A},
\eqref{Case_2B_I_II_V_VI_VII_2},
\eqref{Case_2A_II_3},
\eqref{Case_3B} we obtain 
\begin{equation}
\label{Case all}
\calL_{e_d} f_{\theta}(x) \lesssim - \left (\frac\eps { \delta_D(x)}\wedge 1\right)^{1+\alpha} \frac{  S_{2}^{ \alpha}\vee q^{\alpha/2}}{q^{1+\alpha/2}},\ |x|<r.
\end{equation}
Confronting this with Lemma \ref{S_3_S_4_estimates} we complete the proof.
\end{proof}

We now turn to the proof of the second main proposition in this section.
\begin{proof}[proof of Proposition \ref{F_b_estimate}]
We will consider 2 cases:

{\bf Case 1}. $|x| \in (0, r - \eps)$.

{\bf Case 2}. $|x| \in [r -\eps, r )$.

\vskip 2pt
First, we consider {\bf Case 1}.
 For $y \in \R^d$ put
$$
u(y) = F_{\Theta}(y) -  h(y),
$$
where $h$ is given by (\ref{h_function}) and $F_{\Theta}$ is defined by (\ref{f_definition2}). By (\ref{Theta}) we have for $v \in ((r -\eps)^2, r^2)$ 
$$\frac{1}{r^2 - v} - \Theta(v) = \left(\frac{v- (r - \eps)^2}q\right)^3\frac{1}{r^2 - v}.
$$
Hence for $|y| \in (r - \eps, r)$ we have 
\begin{eqnarray}
\nonumber
u(y) &=&  (r^2 - |y|^2)^{\alpha/2} \left(\varPhi(y) - \frac{1}{r^2 - |y|^2}\right)\\
&=&  -(r^2 - |y|^2)^{\alpha/2}\left(\frac{|y|^2 - (r - \eps)^2}q\right)^3\frac{1}{r^2 - |y|^2}.\label{u_formula1I} 
\end{eqnarray} 
Therefore for $s \in (S_1, S_2)$ we have 
\begin{eqnarray}u(\tilde{x} +s e_d)&=&-(r^2 - |\tilde{x}|^2-s^2)^{\alpha/2-1}\left(\frac{|\tilde{x}|^2+s^2 - (r - \eps)^2}q\right)^3\nonumber\\
&=&-(S_2^2-s^2)^{\alpha/2-1}\left(\frac{s^2-S_1^2}q\right)^3 \label{u_exp}
\end{eqnarray}

 By (\ref{h_function_generator}) for any $y \in B(0,r)$ we have $\calL_{e_d} F_{\Theta}(y) = \calL_{e_d} u(y)$.
 
 Note also that $u(y) = 0$ for $y \in B(0,r - \eps)$. In Case 1 we have $x \in B(0,r - \eps)$ so
\begin{eqnarray*}
\calL_{e_d} F_{\Theta}(x) &=& \calL_{e_d} u(x)\\
&=& \calA_{\alpha} \int_{S_1}^{S_2} \frac{u(\tilde{x} +s e_d)}{|s - x_d|^{1+\alpha}} \, ds
+   \calA_{\alpha} \int_{-S_2}^{-S_1} \frac{u(\tilde{x} +s e_d)}{|s - x_d|^{1+\alpha}}  \, ds\\
&=& \text{I} + \text{II}.
\end{eqnarray*}

Recall that we have assumed that $x_d\ge 0$. It is enough to estimate $\text{I}$. Applying (\ref{u_exp}) we have 

\begin{eqnarray*}
|\text{I}|
&=& \calA_{\alpha} \int_{S_{1}}^{S_{2}} 
(S_2^2-s^2)^{\alpha/2-1}\left(\frac{s^2-S_1^2}q\right)^3\frac1{(s - x_d)^{1+\alpha} }
 \, ds\\
&\le& \calA_{\alpha} \int_{S_{1}}^{S_{2}} 
(S_2^2-s^2)^{\alpha/2-1}\left(\frac{s^2-S_1^2}q\right)^3\frac1{(s - S_1)^{1+\alpha} }
 \, ds\\ 
 &\le& \calA_{\alpha} \int_{S_{1}}^{S_{2}} 
(q-(s^2-S_1^2))^{\alpha/2-1}\left(\frac{s^2-S_1^2}q\right)^3\frac{2^{1+\alpha}s S_2^\alpha}{(s^2 - S_1^2)^{1+\alpha} }\, ds\\
&\le& c\frac{S_{2}^\alpha}{q^{1+\alpha/2}}.
\end{eqnarray*}
The last inequality is obtained by elementary calculations.

Next we additionally assume that  $|x|\le r-2\eps$. In this case we have

$$S_1-x_d=  \frac {S^2_1-x^2_d}{S_1+x_d}\ge \frac18 \frac{\delta_D(x)\eps r}{S_2\eps}\ge c \frac{\delta_D(x) q}{S_2 \eps}. $$

hence
\begin{eqnarray*}
|\text{I}|
&=& \calA_{\alpha} \int_{S_{1}}^{S_{2}} 
(S_2^2-s^2)^{\alpha/2-1}\left(\frac{s^2-S_1^2}q\right)^3\frac1{(s - x_d)^{1+\alpha} }
 \, ds\\
&\le& 4\calA_{\alpha}\frac1{S_2(S_1-x_d)^{1+\alpha}} \int_{S_{1}}^{S_{2}} 
s(S_2^2-s^2)^{\alpha/2-1}
 \, ds\\  
&\le& c\frac{S_{2}^\alpha}{q^{1+\alpha/2}}\left(\frac \eps{\delta_D(x)}\right)^{1+\alpha}.
\end{eqnarray*}

Combining the two estimates we obtain that in the Case 1
$$|\text{I}|\le  c\frac{S_{2}^\alpha}{q^{1+\alpha/2}}\left(1\wedge\frac \eps{\delta_D(x)}\right)^{1+\alpha}$$
yielding
$$\calL_{e_d} F_{\Theta}(x)\ge -c\frac{S_{2}^\alpha}{q^{1+\alpha/2}}\left(1\wedge\frac \eps{\delta_D(x)}\right)^{1+\alpha}.$$
This and Lemma \ref{S_3_S_4_estimates} gives the assertion of the proposition in Case 1.

Now, we consider {\bf Case 2} that is
 $|x| \in [r -\eps, r)$ and for $y \in \R^d$ we define
$$
u(y) = F_{\Theta}(y) -\lambda(y) \varPhi(x)= \lambda(y) \varPhi(y) -\lambda(y) \varPhi(x).
$$
Observe that 
$$
\calL_{e_d} F_{\Theta}(x) = - \tilde{\calA}_{\alpha}  \varPhi(x) 
+ \calL_{e_d} u(x).
$$
Put $p(x) = S_2 - x_d$. Since $u(x) = 0$ we obtain that $\calL_{e_d} u(x)$ is equal to 
\begin{equation*}
\Aa \int_{|t| \ge p(x)} u(x + e_d t) \, \frac{dt}{|t|^{1 + \alpha}}
+ \frac{\Aa}{2}  \int_{|t| < p(x)} \left[u(x + e_d t) + u(x - e_d t) \right] \, \frac{dt}{|t|^{1 + \alpha}}.
\end{equation*}
Put $W_{x} = \{s \in \R: s \in [x_d - p(x), x_d + p(x)]\}$. 
It follows that in Case 2 we have
\begin{eqnarray*}
\calL_{e_d} F_{\Theta}(x) &=& - \tilde{\calA}_{\alpha}  \varphi(x)
+ \calA_{\alpha} \int_{(-S_2,S_2) \setminus W_{x}} \frac{u(\tilde{x} +s e_d)}{|s - x_d|^{1+\alpha}} \, ds\\ 
&+& \frac{\Aa}{2}  \int_{|t| < p(x)} \left[u(x + e_d t) + u(x - e_d t) \right] \, \frac{dt}{|t|^{1 + \alpha}}\\
&=& \text{I} + \text{II} + \text{III}.
\end{eqnarray*}

For $t \in (-p(x), p(x))$ we have
\begin{eqnarray}
\nonumber &&u(x + e_d t) + u(x - e_d t)\\
\nonumber
&=& \left(\lambda(x + t e_d) (\varPhi(x + t e_d) - \varPhi(x)) 
+ \lambda(x - t e_d) (\varPhi(x - t e_d) - \varPhi(x))\right)\\
\nonumber
&=& \lambda(x - t e_d) \left(\varPhi(x + t e_d) + \varPhi(x - t e_d) - 2 \varPhi(x)\right)\\
\nonumber
&+& \left(\lambda(x + t e_d) - \lambda(x - t e_d)  \right) \left(\varPhi(x + t e_d) - \varPhi(x) \right)\\
&\ge&  \left(\lambda(x + t e_d) - \lambda(x - t e_d)  \right) \left(\varPhi(x + t e_d) - \varPhi(x) \right) \label{Case 2L}
\end{eqnarray}
since the term $\lambda(x - t e_d) \left(\varPhi(x + t e_d) + \varPhi(x - t e_d) - 2 \varPhi(x)\right)$ is nonnegative by 
convexity of the function $t\to \varPhi(x + t e_d)$ for $t \in (-p(x), p(x))$.

By (\ref{Case 2L}) and Lemma \ref{estimates_double1}   we obtain
\begin{eqnarray}
\nonumber
\text{III} &=&
\calA_{\alpha}  \int_{|t| < p(x)} \left[u(x + e_d t) + u(x - e_d t) \right] \, \frac{dt}{|t|^{1 + \alpha}}\\
\nonumber
&\ge& \calA_{\alpha}  \int_{|t| < p(x)} \left(\lambda(x + t e_d) - \lambda(x - t e_d)  \right) \left(\varPhi(x + t e_d) - \varPhi(x) \right) \, \frac{dt}{|t|^{1 + \alpha}}\\
\label{Case_2L_III}
&\ge& -c\frac{S_2^\alpha }{q^{1+\alpha/2}}.
\end{eqnarray}

By  Lemma \ref{Case2L} 
\begin{eqnarray}
\label{Case_2L_II}
\text{II} 
 &\ge& -c\frac{S_2^\alpha \vee q^{\alpha/2}}{q^{1+\alpha/2}}.
\end{eqnarray}
and by  Lemma \ref{theta_estimates}
\begin{eqnarray}
\label{Case_2L_I}
\text{I} 
 &\ge& -c \frac1q\ge -c\frac{S_2^\alpha \vee q^{\alpha/2}}{q^{1+\alpha/2}}.
\end{eqnarray}

Then combining   (\ref{Case_2L_III}), (\ref{Case_2L_II}) and  (\ref{Case_2L_I}) we arrive at
$$\calL_{e_d} F_{\Theta}(x)\ge -c\frac{S_2^\alpha \vee q^{\alpha/2}}{q^{1+\alpha/2}} 
= - c \frac{S_2^\alpha \vee q^{\alpha/2}}{q^{1+\alpha/2}} \left(\frac{\eps}{\delta_D(x)} \wedge 1\right)^{1 + \alpha}.$$
This and Lemma \ref{S_3_S_4_estimates} gives the assertion of the proposition in Case 2.
\end{proof}

\section{Appendix}
In the Appendix we provide an explicit construction of a function in $\calG(r,\eps)$.

Fix $r > 0$ and $\eps \in (0,r/4]$. Recall that $q = r^2 - (r- \eps)^2$. Let $K_1$, $K_2$ be constants such that $K_1 q \in (0,r\eps/8)$, 
$K_2 \in (0,r\eps/8)$.

Put 
\begin{eqnarray*}
T_0 &=& (r - \eps)^2,\\
T_1 &=& (r - \eps)^2 + K_1q,\\
T_2 &=& (r - \eps)^2 + K_1q + K_2.
\end{eqnarray*}
We define a function $\theta$ by
\begin{equation}
\label{construction_1}
\theta(x)= \left\{              
\begin{array}{lll}                   
\frac{1}{r^2 - x},&\text{for}& x \in [0,T_0], \\ 
\theta(T_0) + \theta'_{-}(T_0) (x-T_0) + \frac{1}{2}\theta''_{-}(T_0) (x-T_0)^2 - \frac{1}{3q^4 K_1} (x - T_0)^3,&\text{for}& x \in (T_0,T_1], \\ 
\theta(T_1) + \theta'_{-}(T_1) (x-T_1) + \frac{\theta'_{-}(T_1)}{2K_2^3}\left(-2K_2(x-T_1)^3+(x-T_1)^4\right),&\text{for}& x \in (T_1,T_2], \\ 
\theta(T_2),&\text{for}& x \in [T_2,r^2),
\end{array}       
\right. 
\end{equation}
where $\theta'_{-}(T_i)$, $\theta''_{-}(T_i)$ denote the first and second left-hand derivatives of $\theta$ at the points $T_i$. 

At $T_0$ we have the following values of $\theta$, $\theta'_{-}$ and $\theta''_{-}$.
\begin{eqnarray*}
\theta(T_0) &=& 1/q,\\
\theta'_{-}(T_0) &=& 1/q^2,\\
\theta''_{-}(T_0) &=& 2/q^3.
\end{eqnarray*}

On $(T_0,T_1]$ we have the following expressions for $\theta$, $\theta'$ and $\theta''$.
\begin{eqnarray*}
\theta(T_1) &=& \frac{1}{q} \left(1+K_1+\frac{2 K_1^2}{3}\right),\\
\theta'(x) &=& \theta'_{-}(T_0) + \theta''_{-}(T_0) (x-T_0) - \frac{1}{q^4 K_1} (x - T_0)^2,\quad \text{for $x \in (T_0,T_1)$,}\\
\theta'_{-}(T_1) &=& \frac{1}{q^2} \left(1+K_1\right),\\
\theta''(x) &=& \theta''_{-}(T_0) - \frac{2}{q^4 K_1} (x - T_0),\quad \text{for $x \in (T_0,T_1)$,}\\
\theta''_{-}(T_1) &=& 0.
\end{eqnarray*}

Finally, on $(T_1,T_2]$ we have
\begin{eqnarray}
\label{theta_T2_T0}
\theta(T_2) &=& \theta(T_1) + \frac{K_2 \theta'_{-}(T_1)}{2} = \frac{1}{q} \left(1+K_1+\frac{2 K_1^2}{3}\right) + \frac{K_2}{2q^2} \left(1+K_1\right),\\
\nonumber
\theta'(x) &=& \theta'_{-}(T_1) + \frac{\theta'_{-}(T_1)}{2K_2^3}\left(-6K_2(x-T_1)^2+4(x-T_1)^3\right),\quad \text{for $x \in (T_1,T_2)$,}\\
\theta'_{-}(T_2) &=& 0,\\
\nonumber
\theta''(x) &=& \frac{\theta'_{-}(T_1)}{2K_2^3}\left(-12K_2(x-T_1)+12(x-T_1)^2\right),\quad \text{for $x \in (T_1,T_2)$,}\\
\nonumber
\theta''_{-}(T_2) &=& 0.
\end{eqnarray}

%It follows that $\theta \in C^2[0,r^2)$, so (ii) in Definition \ref{Class_A_1} is satisfied. 
Hence $\theta \in C^2[0,r^2)$, so condition (ii) in Definition \ref{Class_A_1} holds.

Let $K>0$ be a constant satisfying $(r-\eps)^2 + qK_1 + K_2 = (r - \eps + K)^2$. We have
\begin{eqnarray*}
K &=& \sqrt{(r - \eps)^2 + (qK_1 + K_2)} - (r - \eps)\\
&=& \frac{qK_1 + K_2}{\sqrt{(r - \eps)^2 + (qK_1 + K_2)} + (r - \eps)}\\
&\le& \frac{r \eps/4}{2(r - \eps)}\\
&<& \frac{\eps}{4},
\end{eqnarray*}
where we used our assumptions: $q K_1 \le \eps/8$, $K_2 \le \eps/8$, $\eps \le r/4$. 
%(\ref{construction_1}) and estimates of the constant $K$ imply that (i) in Definition \ref{Class_A_1} is satisfied. 
By \eqref{construction_1} and the above bound on $K$, condition (i) in Definition \ref{Class_A_1} holds.

%By straightforward computation one easily gets that $\theta' > 0$ on $(T_0,T_1)$ and $(T_1,T_2)$, so (iii) in Definition \ref{Class_A_1} is %satisfied. Similarly, one obtains $\max\{\theta''(x): \, x \in [0,r^2)\} = \theta''((r-\eps)^2)$, so (iv) in Definition \ref{Class_A_1} is %satisfied. Hence, $\theta \in  \calG(r,\eps)$.
A direct computation shows that $\theta'>0$ on $(T_0,T_1)$ and $(T_1,T_2)$, hence condition (iii) holds.
Moreover, one checks that $\max\{\theta''(x):x\in[0,r^2)\}=\theta''((r-\eps)^2)$, so condition (iv) holds. Therefore, $\theta \in  \calG(r,\eps)$.

\begin{lemma}
\label{properties_theta}
Fix $r > 0$ and $\eps \in (0,r/4]$. For any $N > 0$ there exist constants $K_1, K_2$ satisfying $K_1 q \in (0,r \eps/8)$, $K_2 \in (0,r \eps/8)$ and a function $\theta \in \calG(r, \eps)$ defined by (\ref{construction_1}) such that 
\begin{equation}
\label{1N}
K = K_1 q + K_2 \le \frac{\eps}{N}.
\end{equation}
\begin{equation}
\label{1N4}
\|\theta\|_{\infty} - \frac{1}{q} \le \frac{1}{N^{4 + \alpha}}.
\end{equation}
\end{lemma}
\begin{proof}
Indeed, by (\ref{theta_T2_T0}), we have
$$
\|\theta\|_{\infty} - \frac{1}{q} = \theta(T_2) - \theta(T_0) = 
\frac{1}{q} \left(K_1 + \frac{2 K_1^2}{3}\right) + \frac{K_2}{2q^2} \left(1+K_1\right).
$$
Hence for any $N > 0$ there exists sufficiently small $K_1, K_2$ such that $K_1 q \in (0,r \eps/8)$, $K_2 \in (0,r \eps/8)$ and (\ref{1N}), (\ref{1N4}) are satisfied.
\end{proof}

%\textbf{ Acknowledgements.} 

\end{document}